\pdfoutput=1
\newif\ifpersonal
\RequirePackage[l2tabu,orthodox]{nag} 
\documentclass{amsart} 
\usepackage{amsmath,amsthm,amssymb,mathrsfs,mathtools,bm,eucal,tensor} 
\usepackage{microtype,lmodern} 
\usepackage[utf8]{inputenc} 
\usepackage[T1]{fontenc} 
\usepackage{enumitem,comment,braket,xspace,tikz-cd,csquotes} 
\usepackage[centering,vscale=0.7,hscale=0.7]{geometry}
\usepackage[hidelinks]{hyperref}
\usepackage[capitalize]{cleveref}

\allowdisplaybreaks

\numberwithin{equation}{section}
\theoremstyle{plain}
\newtheorem{theorem}[equation]{Theorem}
\newtheorem{lemma}[equation]{Lemma}

\newtheorem*{claim*}{Claim}

\newtheorem{proposition}[equation]{Proposition}

\theoremstyle{definition}
\newtheorem{definition}[equation]{Definition}

\newtheorem{remark}[equation]{Remark}

\ifpersonal
\newcommand{\personal}[1]{\textcolor[rgb]{0,0,1}{(Personal: #1)}}
\else
\newcommand{\personal}[1]{\ignorespaces}
\fi

\providecommand{\abs}[1]{\lvert#1\rvert}

\newcommand{\bbB}{\mathbb B}
\newcommand{\bbC}{\mathbb C}

\newcommand{\bbN}{\mathbb N}

\newcommand{\bbQ}{\mathbb Q}
\newcommand{\bbR}{\mathbb R}
\newcommand{\bbT}{\mathbb T}
\newcommand{\bbU}{\mathbb U}

\newcommand{\bbZ}{\mathbb Z}

\newcommand{\rC}{\mathrm C}
\newcommand{\rH}{\mathrm H}

\newcommand{\rR}{\mathrm R}

\newcommand{\fX}{\mathfrak X}
\newcommand{\fY}{\mathfrak Y}

\newcommand{\cF}{\mathcal F}

\newcommand{\cI}{\mathcal I}

\newcommand{\cM}{\mathcal M}
\newcommand{\cN}{\mathcal N}
\newcommand{\cO}{\mathcal O}

\newcommand{\cX}{\mathcal X}

\newcommand{\bA}{\mathbf A}

\newcommand{\bV}{\mathbf V}
\newcommand{\bW}{\mathbf W}

\newcommand{\bZ}{\mathbf Z}

\makeatletter
\let\save@mathaccent\mathaccent
\newcommand*\if@single[3]{%
	\setbox0\hbox{${\mathaccent"0362{#1}}^H$}%
	\setbox2\hbox{${\mathaccent"0362{\kern0pt#1}}^H$}%
	\ifdim\ht0=\ht2 #3\else #2\fi
}
\newcommand*\rel@kern[1]{\kern#1\dimexpr\macc@kerna}
\newcommand*\widebar[1]{\@ifnextchar^{{\wide@bar{#1}{0}}}{\wide@bar{#1}{1}}}
\newcommand*\wide@bar[2]{\if@single{#1}{\wide@bar@{#1}{#2}{1}}{\wide@bar@{#1}{#2}{2}}}
\newcommand*\wide@bar@[3]{%
	\begingroup
	\def\mathaccent##1##2{%
		\let\mathaccent\save@mathaccent
		\if#32 \let\macc@nucleus\first@char \fi
		\setbox\z@\hbox{$\macc@style{\macc@nucleus}_{}$}%
		\setbox\tw@\hbox{$\macc@style{\macc@nucleus}{}_{}$}%
		\dimen@\wd\tw@
		\advance\dimen@-\wd\z@
		\divide\dimen@ 3
		\@tempdima\wd\tw@
		\advance\@tempdima-\scriptspace
		\divide\@tempdima 10
		\advance\dimen@-\@tempdima
		\ifdim\dimen@>\z@ \dimen@0pt\fi
		\rel@kern{0.6}\kern-\dimen@
		\if#31
		\overline{\rel@kern{-0.6}\kern\dimen@\macc@nucleus\rel@kern{0.4}\kern\dimen@}%
		\advance\dimen@0.4\dimexpr\macc@kerna
		\let\final@kern#2%
		\ifdim\dimen@<\z@ \let\final@kern1\fi
		\if\final@kern1 \kern-\dimen@\fi
		\else
		\overline{\rel@kern{-0.6}\kern\dimen@#1}%
		\fi
	}%
	\macc@depth\@ne
	\let\math@bgroup\@empty \let\math@egroup\macc@set@skewchar
	\mathsurround\z@ \frozen@everymath{\mathgroup\macc@group\relax}%
	\macc@set@skewchar\relax
	\let\mathaccentV\macc@nested@a
	\if#31
	\macc@nested@a\relax111{#1}%
	\else
	\def\gobble@till@marker##1\endmarker{}%
	\futurelet\first@char\gobble@till@marker#1\endmarker
	\ifcat\noexpand\first@char A\else
	\def\first@char{}%
	\fi
	\macc@nested@a\relax111{\first@char}%
	\fi
	\endgroup
}
\makeatother

\newcommand{\oC}{\widebar C}

\newcommand{\oM}{\widebar M}

\newcommand{\RigSHet}{\mathrm{RigSH}_\et}
\newcommand{\RigSm}{\mathrm{RigSm}}
\newcommand{\eff}{\mathrm{eff}}
\newcommand{\PD}{\mathrm{PD}}
\newcommand{\ff}{\mathfrak f}

\newcommand{\HBM}{\mathrm{H}^\mathrm{BM}}
\newcommand{\CBM}{\mathrm{C}^\mathrm{BM}}

\newcommand{\bGm}{\mathbf G_{\mathrm m}}
\newcommand{\longto}{\longrightarrow}
\newcommand{\rsp}{\mathrm{sp}}

\newcommand{\Shv}{\mathrm{Shv}}

\newcommand{\infcat}{$\infty$-category\xspace}

\newcommand{\trunc}{\mathrm{t}_0}

\newcommand{\anL}{\mathbb L\an}

\newcommand{\bfMap}{\mathbf{Map}}

\newcommand{\st}{\mathrm{st}}

\newcommand{\llp}{(\!(}
\newcommand{\rrp}{)\!)}
\newcommand{\an}{^\mathrm{an}}

\newcommand{\et}{\mathrm{\acute{e}t}}
\newcommand{\ev}{\mathrm{ev}}

\newcommand{\inv}{^{-1}}
\newcommand{\id}{\mathrm{id}}

\newcommand{\kanal}{$k$-analytic\xspace}

\newcommand{\oMpre}{\oM^\mathrm{pre}}
\newcommand{\oCpre}{\oC^\mathrm{pre}}

\newcommand{\DM}{Deligne-Mumford\xspace}
\newcommand{\RoM}{\bbR\oM}

\providecommand{\abs}[1]{\lvert#1\rvert}

\newcommand*{\longhookrightarrow}{\ensuremath{\lhook\joinrel\relbar\joinrel\rightarrow}}

\usetikzlibrary{decorations.markings} 
\tikzset{
  closed/.style = {decoration = {markings, mark = at position 0.5 with { \node[transform shape, xscale = .8, yscale=.4] {/}; } }, postaction = {decorate} },
  open/.style = {decoration = {markings, mark = at position 0.5 with { \node[transform shape, scale = .7] {$\circ$}; } }, postaction = {decorate} }
}

\DeclareMathOperator{\iHom}{\underline{Hom}}

\DeclareMathOperator{\Map}{Map}

\DeclareMathOperator{\Sp}{Sp}

\DeclareMathOperator{\Spec}{Spec}
\DeclareMathOperator{\Spf}{Spf}

\DeclareMathOperator{\Sym}{Sym}
\DeclareMathOperator{\val}{val}

\DeclareMathOperator*{\colim}{colim}

\begin{document}
\title{Non-archimedean Gromov-Witten invariants}

\author{Mauro PORTA}
\address{Mauro PORTA, Institut de Recherche Mathématique Avancée, 7 Rue René Descartes, 67000 Strasbourg, France}
\email{porta@math.unistra.fr}

\author{Tony Yue YU}
\address{Tony Yue YU, Department of Mathematics M/C 253-37, California Institute of Technology, 1200 E California Blvd, Pasadena, CA 91125, USA}
\email{yuyuetony@gmail.com}
\date{September 26, 2022}

\begin{abstract}
Motivated by mirror symmetry and the enumeration of holomorphic disks, we construct the theory of Gromov-Witten invariants in the setting of non-archimedean analytic geometry.
We build on our previous works on derived non-archimedean geometry and non-archimedean quantum K-invariants, as well as recent developments of rigid analytic motives and virtual fundamental classes in derived geometry.
Our approach gives also a new perspective for the classical algebraic case.
\end{abstract}

\maketitle

\tableofcontents

\section{Introduction} \label{sec:intro}

The Gromov-Witten invariants of a smooth complex projective variety $X$ are rational numbers which intuitively count closed algebraic curves in $X$ of given genus, homology class and incidence conditions.
They are of great interest in mathematics because they provide the first rigorous and general approach to the ancient problem of curve counting in algebraic varieties (see \cite{Pandharipande_13/2_ways}).
More importantly, they are essential for the mathematical formulation of mirror symmetry (see \cite{Hori_Mirror_symmetry}).

Recall that mirror symmetry is a conjectural duality between Calabi-Yau varieties.
Roughly, it predicts that given any Calabi-Yau variety $X$, there is a mirror Calabi-Yau variety $Y$, such that we have an ever-growing list of deep geometric relations between $X$ and $Y$, involving Hodge structures, Gromov-Witten invariants, Fukaya categories, derived categories of coherent sheaves, SYZ torus fibrations, and so on.
A more careful study on the origin of mirror symmetry reveals that mirror symmetry should not be considered as a duality between individual Calabi-Yau varieties, but rather a duality between maximally degenerating families of Calabi-Yau varieties (see \cite{Strominger_Mirror_symmetry_is_T-duality,Kontsevich_Homological_mirror_symmetry,Gross_Large_complex_structure_limits}).
In general, given any algebraic family of varieties $\cX_t$ over a small punctured disk in $\bbC$ with coordinate $t$, we obtain a single algebraic variety $X$ defined over the field $\bbC\llp t\rrp$ of formal Laurent series.
Note that $\bbC\llp t\rrp$ is a non-archimedean field with respect to the norm given by the $t$-adic valuation.
So it makes sense to study $X$ not only as an algebraic variety, but also as a non-archimedean analytic space, after applying non-archimedean analytification $X\an$.
This observation opens exciting new ways for studying the family $\cX_t$ thanks to substantial recent developments on the subject of non-archimedean geometry (see \cite{Conrad_Several_approaches,Berkovich_Spectral_theory,Berkovich_Etale_cohomology,Huber_Etale_cohomology}).

More concretely, we now have a rigorous foundation of SYZ (Strominger-Yau-Zaslow) torus fibration in the non-archimedean setting (see \cite{Nicaise_Xu_Yu_The_non-archimedean_SYZ_fibration}), whose original archimedean version is still far beyond reach.
This will also facilitate the mirror reconstruction problem.
It is noted in \cite[\S 1.2]{Kontsevich_Affine_structures} that heuristic ideas for mirror reconstruction were proposed independently by Fukaya and Kontsevich more than twenty years ago, but their realization was hindered by a poor understanding of the SYZ torus fibration, and by the lack of knowledge of instanton corrections given by counts of curves with boundaries on SYZ torus fibers.
The non-archimedean approach is supposed to solve both problems simultaneously.
The proof of concept has been worked out in the case of affine log Calabi-Yau varieties (see \cite{Yu_Enumeration_of_holomorphic_cylinders_I,Yu_Enumeration_of_holomorphic_cylinders_II,Keel_Yu_The_Frobenius}),
and has found applications beyond mirror symmetry, towards cluster algebras in representation theory (\cite{Keel_Yu_The_Frobenius}) and the compactification of moduli spaces (\cite{Hacking_Keel_Yu_Secondary_fan}).

Nevertheless, due to the lack of a robust theory of non-archimedean enumerative geometry, the authors were obliged to restrict to naive curve counts in the above works.
Although naive curve counts are highly intuitive, they are delicate to work with, and this also hampers any generalizations beyond affine log Calabi-Yau varieties.
Therefore, it has been desired for a long time to develop an analog of Gromov-Witten theory in non-archimedean analytic geometry.

In order to set up a new enumerative theory, there are in general two issues to be resolved: compactness and transversality.
For symplectic Gromov-Witten theory, the compactness is based on the famous Gromov compactness theorem (see \cite{Gromov_Pseudoholomorphic_curves,Pansu_Compactness,Ye_Gromov's_compactness,Hummel_Gromov's_compactness}).
Its algebraic formulation is worked out in \cite{Kontsevich_Enumeration,Fulton_Notes_on_stable_maps}, and the non-archimedean analog is established in \cite{Yu_Gromov_compactness}.
The issue of transversality means that the moduli space of curves in a given target space can be singular and have larger than expected dimension.
So the correct enumerative invariants should not be given by the actual fundamental class of the moduli space, but rather a ``virtual fundamental class'' (see \cite[\S 1.4]{Kontsevich_Enumeration}).
A proposal for the construction of virtual fundamental classes was made by Kontsevich in loc.\ cit., using hypothetical notions of super-scheme and quasi-manifold.
Although these ideas were precursors of modern derived algebraic geometry, they were not successfully carried out at that time due to the lack of foundations of higher category theory and higher algebra.
Nevertheless, they quickly inspired a variety of methods for the construction of virtual fundamental classes in symplectic geometry (see \cite{Fukaya_Arnold_conjecture,McDuff_Notes_on_Kuranishi_atlases,Ruan_Virtual_neighborhoods,Hofer_Applications_of_polyfold_theory_I,Pardon_An_algebraic_approach_to_virtual_fundamental_cycles}) and in algebraic geometry (see \cite{Li_Virtual_moduli,Behrend_Intrinsic_normal_cone}).

Unfortunately, none of the existing methods in symplectic geometry or algebraic geometry can be easily carried over to non-archimedean analytic geometry.
Thankfully, the theory of derived algebraic geometry is now well-established (see \cite{DAG-V,Lurie_SAG,HAG-II}).
Our approach towards the issue of transversality is to develop an analog of derived geometry in the non-archimedean setting (see \cite{Porta_Yu_Derived_non-archimedean_analytic_spaces,Porta_Yu_Representability_theorem}).
Using the representability theorem, we endow the moduli stack of non-archimedean stable maps with a derived structure (see \cite{Porta_Yu_Representability_theorem,Porta_Yu_Derived_Hom_spaces,Porta_Yu_Non-archimedean_quantum_K-invariants}), which is supposed to retain all information of the underlying derived moduli problem (see \cite{DAG-X}), in particular, the original enumerative problem.

Now the question becomes: how do we extract numerical enumerative invariants from the derived structure.
In algebraic geometry, we have two types of numerical invariants built from Gromov-Witten theory: the classical cycle-theoretic invariants via Chow groups (see \cite{Li_Virtual_moduli,Behrend_Gromov-Witten_invariants}), and the K-theoretical invariants  (see \cite{Givental_On_the_WDVV,Lee_Quantum_K-theory_I}).
K-theory works similarly in non-archimedean geometry, and we deduced the K-theoretical invariants from the derived structure and established all the expected properties in \cite[\S 8]{Porta_Yu_Non-archimedean_quantum_K-invariants}.
We were not able to deduce the cycle-theoretic invariants in loc.\ cit., because intersection theory for Chow groups as in \cite{Fulton_Intersection_theory} does not work generally in analytic geometry, e.g.\ we may not have enough cycles to arrange into transverse positions.

The purpose of this paper is to circumvent this problem and build the cycle-theoretic numerical invariants from the derived structure.
The insight is that Chow groups may not be the right objects for intersection theory in non-archimedean analytic geometry.
Instead, we will use the theory of rigid analytic motives developed by Ayoub \cite{Ayoub_Motifs}.
Although the exact relations between cycles and motivic homology groups are not yet well understood, the six functor formalism for rigid analytic motives (see \cite{Ayoub_The_six}) is already sufficiently powerful for the purpose of intersection theory at the level of motivic homologies; in fact, it gives not only an intersection theory, but a \emph{bivariant} intersection theory in the sense of Fulton and MacPherson (\cite{Fulton_Categorical_framework}) in non-archimedean geometry.

It is plausible to use other cohomology theories in non-archimedean geometry for the purpose of extracting numerical invariants, such as étale cohomology (\cite{Berkovich_Etale_cohomology,Huber_Etale_cohomology,de_Jong_Etale_cohomology}) or de Rham cohomology (\cite{Schneider_The_cohomology,Grosse-Klonne_De_Rham_cohomology}).
However, one disadvantage is that they do not have rational coefficients.
So the numerical invariants (i.e.\ the counts of curves) will be $l$-adic numbers (in the case of étale cohomology) or elements of the ground non-archimedean field (in the case of de Rham cohomology), neither of which is satisfactory.
In the case of $\bbC\llp t\rrp$ as the ground field, there is an étale cohomology theory with integer coefficients via vanishing cycles for formal schemes by Berkovich \cite{Berkovich_Complex_analytic_vanishing_cycles_for_formal_schemes}, which may be used to overcome the above disadvantage, provided that we develop sufficient functorial properties of the theory.

Another essential tool we use in this paper is the deformation to the shifted tangent bundle by Khan and Rydh \cite{Khan_Virtual_Cartier_divisors}, which is a derived analog of the deformation to the normal cone.
This deformation gives rise to the desired virtual fundamental class, as constructed in Khan \cite{Khan_Virtual_fundamental_classes}.

The resulting virtual fundamental classes satisfy all the expected properties parallel to the classical Behrend-Manin axioms in \cite{Behrend_Stacks_of_stable_maps}.

\begin{theorem}[see \cref{sec:axioms}] \label{thm:relations_vfc}
	Let $S$ be a rigid \kanal space and $X$ a rigid \kanal space smooth over $S$.
	Let $\RoM(X/S,\tau,\beta)$ denote the derived moduli stack of $(\tau,\beta)$-marked stable maps into $X/S$ associated to an A-graph $(\tau,\beta)$.
	Let
	\[[\RoM(X/S,\tau,\beta)]\in\HBM_d\big(\RoM(X/S,\tau,\beta)/S,\bbQ_S(2d)\big)\]
	be the associated virtual fundamental class in the motivic Borel-Moore homology, where $d$ denotes the virtual dimension.
	It satisfies the following equalities with respect to elementary operations on A-graphs:
	
	\medskip\noindent
	(1) Mapping to a point:
	Let $(\tau,0)$ be any A-graph where $\beta$ is 0.
	We have $\oM(X/S,\tau,0)\simeq\oM_\tau\times X$, and
	\[[\RoM(X/S,\tau,0)]=c_{g(\tau)\dim X/S}\big(\rR^1\pi_*\cO_{\oC_\tau}\boxtimes\bbT\an_{X/S}\big)\cap[\oM(X/S,\tau,0)],\]
	where $\pi\colon\oC_\tau\to\oM_\tau$ and $c_\cdot(\cdot)$ denotes the Chern class.
	
	\medskip\noindent
	(2) Products: Let $(\tau_1, \beta_1)$ and $(\tau_2, \beta_2)$ be two A-graphs.
	We have
	\[[\RoM(X/S,\tau_1\sqcup\tau_2,\beta_1\sqcup\beta_2)]=[\RoM(X/S, \tau_1, \beta_1)] \boxtimes [\RoM(X/S, \tau_2, \beta_2)].\]
	
	\medskip\noindent
	(3) Cutting edges: Let $(\sigma,\beta)$ be an A-graph obtained from $(\tau,\beta)$ by cutting an edge $e$ of $\tau$.
	Let $v,w$ be the two tails of $\sigma$ created by the cut.
	We have a derived pullback diagram
	\[ \begin{tikzcd}
		\RoM(X/S, \tau, \beta) \rar \dar{\ev_e} & \RoM(X/S, \sigma, \beta) \dar{\ev_v \times \ev_w} \\
		X \rar{\Delta} & X \times_S X ,
	\end{tikzcd} \]
	and
	\[[\RoM(X/S, \tau, \beta)]=\Delta^![\RoM(X/S, \sigma, \beta)].\]
	
	\medskip\noindent
	(4) Universal curve:
	Let $(\sigma,\beta)$ be an A-graph obtained from $(\tau,\beta)$ by forgetting a tail $t$ attached to a vertex $w$.
	Let
	\[\pi\colon\RoM(X/S,\tau,\beta)\to\RoM(X/S,\sigma,\beta)\]
	denote the forgetful map, and $\oCpre_w\to\oMpre_\sigma$ the universal curve corresponding to $w$.
	We have a derived pullback diagram
	\[ \begin{tikzcd}
		\RoM( X/S, \tau, \beta ) \rar{\pi} \dar & \RoM( X/S, \sigma, \beta ) \dar \\
		\oCpre_w \rar & \oMpre_\sigma ,
	\end{tikzcd} \]
	and
	\[[\RoM(X/S, \tau, \beta)]=\pi^![\RoM(X/S, \sigma, \beta)].\]

	\medskip\noindent
	(5) Forgetting tails:
	Following the context in (4), we have a derived pullback diagram
	\[ \begin{tikzcd}
		\RoM(X/S, \tau, \beta) \rar{\Psi} \dar & \oM_\tau \times_{\oM_\sigma}  \RoM(X/S, \sigma, \beta) \dar \\
		\oCpre_w \rar & \oM_\tau \times_{\oM_\sigma} \oMpre_\sigma ,
	\end{tikzcd} \]
	and
	\[\Psi_*[\RoM(X/S, \tau, \beta)]=\Phi^![\RoM(X/S, \sigma, \beta)],\]
	where $\Phi\colon\oM_\tau\to\oM_\sigma$ is the forgetting-tails map for pointed stable curves.
	
	\medskip\noindent
	(6) \textbf{Contracting edges:}
	Let $(\sigma,\beta)$ be an A-graph where $\sigma$ is obtained from a modular graph $\tau$ by contracting an edge (possibly a loop) $e$.
	Let $\beta_j, j\in J$ be all possible curve classes on $\tau$ that project to $\beta$ under the contraction.
	We have a derived pullback diagram
	\[\begin{tikzcd}
		\coprod_j \RoM(X/S,\tau,\beta_j) \rar{\Psi} \dar& \oM_\tau\times_{\oM_\sigma}\RoM(X/S,\sigma,\beta) \dar \\
		\oMpre_\tau \rar & \oM_\tau \times_{\oM_\sigma} \oMpre_\sigma ,
	\end{tikzcd}\]
	and
	\[\sum_j\Psi_*[\RoM(X/S, \tau, \beta_j)]=\Phi^![\RoM(X/S, \sigma, \beta)] ,\]
	where $\Phi\colon\oM_\tau\to\oM_\sigma$ is the contracting-edges map for pointed stable maps.
\end{theorem}

We summarize the functorial properties of bivariant motivic Borel-Moore homology with respect to pullback, pushforward, Gysin pullback, Gysin pushforward and various products in \cref{sec:motivic_homology,sec:vfc}.
Combining with the above theorem, we deduce the analogous Behrend-Manin axioms for the numerical invariants associated to the virtual fundamental classes, see \cref{sec:numerical_GW}.

In the last section of the paper, we introduce non-archimedean Gromov-Witten invariants with naive tangency conditions.
Consider a proper smooth \kanal space $X$ and an A-graph $(\tau,\beta)$.
For every tail $i$, we fix an lci closed analytic subspace $Z_i\subset X$ and a positive integer $m_i$.
Then we define Gromov-Witten counts of curves whose $i$-th marked point meets $Z_i$ with tangency at least $m_i$.
We prove a list of properties for such counts parallel to the list in \cref{sec:numerical_GW}.
We call it \emph{naive tangency conditions} because the theories of relative Gromov-Witten invariants (\cite{Ionel_Relative_Gromov-Witten_invariants,Li_A_degeneration_formula}) and logarithmic Gromov-Witten invariants (\cite{Chen_Stable_logarithmic_maps_I,Abramovich_Stable_logarithmic_maps_II,Gross_Logarithmic_Gromov-Witten_invariants}) are more sophisticated ways of treating tangency conditions.
Their relations with our naive tangency conditions require further investigation.
Our main motivation for introducing naive tangency conditions is for the construction of structure constants and theta functions in mirror symmetry generalizing \cite{Keel_Yu_The_Frobenius}.
Thanks to the flexibility of non-archimedean analytic curves compared with logarithmic curves, we are hopeful that naive tangency conditions will be sufficient for our purposes.
It is also speculated that one may obtain a simplification of the degeneration formula via naive tangency conditions.

\medskip
Although the whole paper is phrased in the language of non-archimedean analytic geometry in view of the future applications, the theory can be carried over verbatim to algebraic geometry.

Let us summarize three main advantages of our derived approach towards Gromov-Witten invariants versus more classical treatments:

\smallskip\noindent
(1) We mentioned above the two general issues for establishing any new enumerative theory: compactness and transversality.
In the classical approach, the solution for transversality, i.e.\ the construction of virtual fundamental classes, is intertwined with compactness.
In our derived approach, the two issues are now completely separate.
We construct the virtual fundamental classes in motivic Borel-Moore homology, without any assumptions on the compactness of moduli stacks.
This will be important in our applications towards instanton corrections in non-archimedean mirror symmetry, because we will not be able to guarantee compactness in the beginning, instead the compactness will only be attained under further conditions on the associated tropical curves (see \cite{Yu_Enumeration_of_holomorphic_cylinders_I,Keel_Yu_The_Frobenius}).

\smallskip\noindent
(2) Our theory is relative while the classical approach is absolute.
Our construction of non-archimedean Gromov-Witten virtual fundamental classes works for any rigid analytic space $X$ that is smooth over an \emph{arbitrary} rigid analytic space $S$.
This provides flexibility for studying non-archimedean Gromov-Witten invariants in families.
In particular, it will aide the proof of the invariance of counts of non-archimedean holomorphic disks with varying tail conditions, generalizing \cite[\S 13]{Keel_Yu_The_Frobenius}.

\smallskip\noindent
(3) The derived approach is conceptually more intuitive, while technically more involved.
The geometric relations of the derived moduli stacks of stable maps with respect to elementary operations on graphs are as simple as the ones for the underived moduli stacks, nevertheless, they already contain all the subtle information involving virtual fundamental classes and the underlying enumerative problem.
The deductions of Behrend-Manin axioms in \cref{sec:axioms,sec:numerical_GW} are now straightforward applications of various functorial properties of motivic (co)homology.

\bigskip
\paragraph{\textbf{General references}}

We refer to \cite{Lurie_HTT,Lurie_Higher_algebra} for background on $\infty$-categories,
to \cite{Bosch_Non-Archimedean_analysis,Berkovich_Spectral_theory,Berkovich_Etale_cohomology,Huber_Etale_cohomology} for non-archimedean analytic geometry,
to \cite{Mazza_Lecture_notes_on_motivic_cohomology,Cisinski_Triangulated} for motives in algebraic geometry,
to \cite{Ayoub_Motifs,Ayoub_The_six} for rigid analytic motives,
to \cite{Lurie_SAG,HAG-II} for derived algebraic geometry,
to \cite{Porta_Yu_Higher_analytic_stacks,Porta_Yu_Derived_non-archimedean_analytic_spaces,Porta_Yu_Representability_theorem,Porta_Yu_Derived_Hom_spaces} for derived non-archimedean analytic geometry,
to \cite{Kontsevich_Gromov-Witten_classes,Behrend_Stacks_of_stable_maps,Li_Virtual_moduli,Behrend_Gromov-Witten_invariants} for algebraic Gromov-Witten theory,
and to \cite{Porta_Yu_Non-archimedean_quantum_K-invariants} for derived stacks of stable maps.

We fix $k$ a complete non-archimedean field with nontrivial valuation throughout the paper.

\bigskip
\paragraph{\textbf{Acknowledgments}}

We are indebted to Joseph Ayoub, Martin Gallauer and Alberto Vezzani for the powerful theory of rigid analytic motives and six operations.
We would like to thank Adeel Khan for his inspiring work on virtual fundamental classes, and for numerous detailed technical discussions, without which our paper would be impossible.
Our long-term project where this paper belongs has received great support and invaluable insights from Denis Auroux, Vladimir Berkovich, Federico Binda, Antoine Chambert-Loir, Antoine Ducros, Tom Graber, Mark Gross, Walter Gubler, Benjamin Hennion, Ludmil Katzarkov, Sean Keel, Maxim Kontsevich, Gérard Laumon, Y.P.\ Lee, Jacob Lurie, Etienne Mann, Tony Pantev, Francesco Sala, Paul Seidel, Carlos Simpson, Georg Tamme, Bertrand To\"en and Gabriele Vezzosi.
The authors would also like to thank each other for the joint effort.

\section{Rigid analytic motives for stacks} \label{sec:motives}

In this section, we review the theory of motives for rigid analytic spaces following Ayoub et al.\ \cite{Ayoub_Motifs,Ayoub_The_six}, and then explain an extension of the theory to rigid analytic stacks.

Fix a rigid \kanal space $S$.
Let $\RigSm/S$ denote the category of rigid \kanal spaces smooth over $S$, and $\Shv_\et(\RigSm/S,\bbQ)$ the \infcat of étale sheaves of derived $\bbQ$-modules over $\RigSm/S$.
Let
\[\bbQ_\et(-)\colon\RigSm/S\to\Shv_\et(\RigSm/S,\bbQ)\]
denote the Yoneda embedding composed with sheafification.
Let
\[L_{\bbB^1}\colon\Shv_\et(\RigSm/S,\bbQ)\longto\RigSHet^\eff(S,\bbQ)\]
be the localization with respect to the maps of the form $\bbQ_\et(\bbB^1_X)\to\bbQ_\et(X)$ for all $X\in\RigSm/S$, and their desuspensions, where $\bbB^1_X$ denotes the closed unit disk relative over $X$.

The \infcat $\RigSHet^\eff(S,\bbQ)$ has a natural monoidal structure inherited from derived $\bbQ$-modules.
Let $\bbU^1_S$ denote the unit circle relative over $S$, and $T_S$ the image by $L_{\bbB^1}$ of the cofiber of the split inclusion $\bbQ_\et(S)\to\bbQ_\et(\bbU^1_S)$ induced by the unit section.
We define the inversion
\[\Sigma^\infty_T\colon \RigSHet^\eff(S,\bbQ)\longto\RigSHet(S,\bbQ)\coloneqq\RigSHet^\eff(S,\bbQ)[T_S^{-1}].\]
Objects of $\RigSHet(S,\bbQ)$ are called \emph{rigid analytic motives} over $S$.
For any rigid analytic space $X$ smooth over $S$, we denote $M^\eff(X)\coloneqq \Sigma^\infty_T(L_{\bbB^1}(\bbQ_\et(X)))$, called the \emph{motive} of $X$.
We denote by $1_S$ the monoidal unit of $\RigSHet(S,\bbQ)$.
For any $M\in\RigSHet(S,\bbQ)$ and $n\in\bbZ$, we denote by $M(n)$ the Tate twist given by tensoring with $T_S^{\otimes n}$.

The theory of six functors for rigid analytic motives is developed in Ayoub-Gallauer-Vezzani \cite{Ayoub_The_six}.
The \infcat $\RigSHet(S,\bbQ)$ have adjoint bifunctors $(\otimes,\iHom)$, and for any morphism $f\colon T\to S$ of rigid \kanal spaces, we have adjunctions
\begin{align*}
	f^* &\colon\RigSHet(S,\bbQ)\leftrightarrows\RigSHet(T,\bbQ)\colon f_*\\
	f_! &\colon\RigSHet(T,\bbQ)\leftrightarrows\RigSHet(S,\bbQ)\colon f^!.
\end{align*}
They satisfy the standard Grothendieck six functors formalism (see \cite[A.5.1]{Cisinski_Triangulated}).

We can extend rigid analytic motives to (higher) \kanal stacks (see \cite[\S 3.3]{Porta_Yu_Higher_analytic_stacks}) via a right Kan extension.
More precisely, given any \kanal stack $S$, the \infcat $\RigSHet(S,\bbQ)$ is identified with the limit of $\RigSHet(S',\bbQ)$ over the smooth site of $S$.
The six functors for rigid analytic motives extend also to \kanal stacks, following \cite[A.2]{Khan_Virtual_fundamental_classes}.
The only uncertainty is the invertibility of the natural transformation $\alpha_f\colon f_!\to f_*$ when $f$ is proper but non-representable.
We expect it to hold but do not yet have a proof.
Fortunately, the following special case will be sufficient for non-archimedean Gromov-Witten theory at least when the ground field $k$ has residue characteristic zero.

\begin{lemma} \label{lem:alpha_f}
	Let $\ff \colon \fX \to \fY$ be a proper morphism of formal \DM stacks locally finitely presented over $k^\circ$, the ring of integers of $k$.
	Let $f \colon X \to Y$ denote the induced map between the generic fibers.
	For any $\cF\in\RigSHet(X,\bbQ)$, the morphism
	\[\alpha_f\colon f_!(\cF)\to f_*(\cF)\]
	induced by the natural transformation $\alpha_f$ is an equivalence.
\end{lemma}
\begin{proof}
		The question being étale local on the target, we can assume that $\fY$ is affine.
	Let $\abs{\fX_s}$ and $\abs{\fX}$ denote the coarse moduli spaces.
	Then the map $\ff\colon\fX\to\fY$ factors as
	$\fX \xrightarrow{\mathfrak p} \abs{\fX} \xrightarrow{\overline{\ff}} \fY$, so $f \colon X \to Y$ factors as $X \xrightarrow{p} \abs{\fX}_\eta \xrightarrow{\widebar{f}} Y$.
	By \cite[Proposition 3.6]{Abramovich_Tame_stacks}, étale locally over $\abs{\fX_s}$, the \DM stack $\fX_s$ is isomorphic to the quotient of an affine scheme by a finite group.
	Using the equivalence between the étale sites of $\abs{\fX_s} \simeq \abs{\fX}_s$ and of $\abs{\fX}$, we deduce that étale locally over $\abs{\fX}_\eta$, the \kanal \DM stack $X = \fX_\eta$ is isomorphic to the quotient of an affinoid space by a finite group.

	Let us consider such a quotient $q\colon U \to V\coloneqq U/G$ and $s\colon V\to\abs{V}$ the projection to its course moduli space.
	Since $q$ is finite, and $\abs{V}$ is affinoid by \cite[Proposition 6.3.3/3]{Bosch_Non-Archimedean_analysis}, the situation is relatively algebraic over $\abs{V}$.
	Therefore, by the compatibility with analytification (see \cite[\S 4.6]{Ayoub_The_six}), we deduce from \cite[Theorem A.7]{Khan_Virtual_fundamental_classes} that $\alpha_s$ is an equivalence.
	Therefore, $\alpha_p$ is an equivalence.
	
	By \cite[\S B.2]{Conrad_Spreading-out}, the generic fiber $\abs{\fX}_\eta$ of the formal algebraic space $\abs{\fX}$ is a \kanal space.
	Since $\ff$ is proper, so are $\overline{\ff}$ and $\widebar f$.
	We deduce that $\alpha_{\widebar f}$ is an equivalence.
	As $f_* = \widebar{f}_* \circ p_*$, the conclusion follows.
\end{proof}

Finally, we extend rigid analytic motives and six functors to derived \kanal stacks (see \cite[Definition 7.2]{Porta_Yu_Representability_theorem}) by composing with the truncation functor $\trunc$ from derived \kanal stacks to (underived) \kanal stacks.
The rationale behind this quick definition is that the motives as a universal cohomology theory should not be affected by any derived structures.
This is confirmed in the algebraic case by \cite{Khan_Motivic_homotopy_theory_in_derived_algebraic_geometry}.

\section{Bivariant motivic Borel-Moore chains and their products} \label{sec:motivic_homology}

In this section, we introduce bivariant motivic Borel-Moore chains, their functorialities and product operations on them.
Then we deduce motivic Borel-Moore homologies, motivic cohomologies and their properties.
The material is parallel to the algebraic theory of motivic cohomologies (see \cite{Mazza_Lecture_notes_on_motivic_cohomology}, \cite{Deglise_Bivariant_theories} and \cite[\S 2]{Khan_Virtual_fundamental_classes}), but we put more emphasis on the various product operations for the purpose of applications in later sections.

\begin{definition}\label{def:bivariant_chains}
	Let
	\[ \begin{tikzcd}[column sep = small]
		X \arrow{rr}{f} \arrow{dr}[swap]{a} & & Y \arrow{dl}{b} \\
		& S
	\end{tikzcd} \]
	be a triangle of derived \kanal stacks, and let $\cF \in \RigSHet(S,\bbQ)$ be a rigid analytic motive.
	We define the \emph{motivic bivariant Borel-Moore chain} of $f/S$ with coefficients in $\cF$ by
	\[ \CBM_\bullet(f/S,\cF) \coloneqq a_* f^! b^*(\cF) \in \RigSHet(S,\bbQ),\]
	and the \emph{motivic bivariant Borel-Moore homology} of $f/S$ with coefficients in $\cF$ by
	\[\HBM_s\big(f/S,\cF(r)\big)\coloneqq \pi_0 \Map_{\RigSHet(S,\bbQ)}\big(1_S(r)[s],\CBM_\bullet(f/S,\cF)\big),\qquad r,s\in\bbZ.\]
	We will omit the word ``bivariant'' for simplicity.
	When the base $S$ is clear from the context, we will equally denote the above objects by $\CBM_\bullet(X/Y,\cF)$ and $\HBM_s(X/Y,\cF(r))$.
\end{definition}

Here are two special cases of the above definition.

\begin{definition}
	Let $a\colon X\to S$ be a morphism of derived \kanal stacks and $\cF\in\RigSHet(S,\bbQ)$.
	We define:
	\begin{enumerate}[leftmargin=*, itemsep=1ex]
		\item the \emph{motivic Borel-Moore chain} of $X/S$ with coefficients in $\cF$ by
		\[ \CBM_\bullet(X/S,\cF) \coloneqq \CBM_\bullet(a/S,\cF) \simeq a_* a^!(\cF) \in \RigSHet(S,\bbQ), \]
		\item the \emph{motivic Borel-Moore homology} of $X/S$ with coefficients in $\cF$ by
		\[\HBM_s\big(X/S,\cF(r)\big)\coloneqq \pi_0 \Map_{\RigSHet(S,\bbQ)}\big(1_S(r)[s],\CBM_\bullet(X/S,\cF)\big),\qquad r,s\in\bbZ,\]
		\item the \emph{motivic cochain of $X/S$} with coefficients in $\cF$ by
		\[ \rC^\bullet(X/S,\cF) \coloneqq \CBM_\bullet(\id_X/S) \simeq a_* a^*(\cF) \in \RigSHet(S,\bbQ) ,\]
		\item the \emph{motivic cohomology} of $X$ with coefficients in $\cF$ by
		\[\rH^s\big(X,\cF(r)\big)\coloneqq \pi_0\Map_{\RigSHet(S,\bbQ)}\big(1_S(-r)[-s],\rC^\bullet(X/S,\cF)\big),\qquad r,s\in\bbZ,\]
		which is independent of the base $S$.
	\end{enumerate}
\end{definition}

\begin{definition}
	Let
	\[ \begin{tikzcd}[column sep = small]
		X_2 \arrow{rr}{f_2} \dar{g} & & Y_2 \dar{h} \\
		X_1 \arrow{rr}{f_1} \arrow{dr}[swap]{a_1} & & Y_1 \arrow{dl}{b_1} \\
		& S
	\end{tikzcd} \]
	be a commutative diagram of derived \kanal stacks.
	Set $a_2 \coloneqq a_1 \circ g$ and $b_2 \coloneqq a_2 \circ h$.
	\begin{enumerate}[wide=0pt, itemsep=1ex]
		\item Assume that the square is a pullback.
		The \emph{chain-level pullback}
		\[ h^* \colon \CBM_\bullet(f_1/S,\cF) \longrightarrow \CBM_\bullet(f_2/S,\cF) \]
		is defined as the composition of
		\[ a_{1*} f_1^! b_1^*(\cF) \xrightarrow{\eta_*} a_{1*} f_1^! h_* h^* b_1^*(\cF) \simeq a_{1*} g_* f_2^! h^* b_1^*(\cF) \simeq a_{2*} f_2^! b_2^*(\cF) , \]
		where $\eta_*$ denotes the unit of the adjunction $h^* \dashv h_*$.
		
		\item Assume that $g$ is proper and $h$ is étale.
		The \emph{chain-level proper pushforward}
		\[ g_* \colon \CBM_\bullet(f_2/S,\cF) \longrightarrow \CBM_\bullet(f_1/S,\cF) \]
		is defined as the composition of
		\[ a_{2*} f_2^! b_2^*(\cF) \xrightarrow{\eta_!} a_{2*} f_2^! h^! h_! b_2^*(\cF) \simeq a_{1*} g_* g^! f_1^! h_! h^* b_1^*(\cF) \simeq a_{1*} g_! g^! f_1^! h_! h^! b_1^*(\cF) \to a_{1*} f_1^! b_1^*(\cF), \]
		where $\eta_!$ denotes the unit of the adjunction $h_!\dashv h^!$.
	\end{enumerate}
\end{definition}

\begin{remark} \label{rem:motivic_homology_functoriality}
	As special cases of the above definition, we obtain:
	\begin{enumerate}[wide=0pt, itemsep=1ex]
		\item when $f_1 = \id_{X_1}$ and $f_2 = \id_{X_2}$, the \emph{pullback} morphisms
		\begin{align*}
		g^* &\colon \rC^\bullet(X_1/S,\cF) \longrightarrow \rC^\bullet(X_2/S,\cF), \\
		g^* &\colon \rH^s(X_1/S,\cF(r)) \longrightarrow \rH^s(X_2/S,\cF(r)) ;
		\end{align*}
		
		\item when $h = b_1 = \id_{S}$ and $g$ is proper, the \emph{proper pushforward} morphisms
		\begin{align*}
		g_* &\colon \CBM_\bullet(X_2/S,\cF) \longrightarrow \CBM_\bullet(X_1/S,\cF) , \\
		g_* &\colon \HBM_s(X_2/S,\cF(r)) \longrightarrow \HBM_s(X_1/S,\cF(r)) ;
		\end{align*}
		
		\item \label{rem:motivic_homology_functoriality:pushforward_homology_to_cohomology}
		when $h = f_1 = \id_{Y_1}$ and $g$ is proper, the \emph{proper pushforward} morphisms
		\begin{align*}
		g_* &\colon \CBM_\bullet(f_2/S) \longrightarrow \rC^\bullet(Y_1/S) , \\
		g_* &\colon \HBM_s(f_2/S,\cF(r)) \longrightarrow \rH^{-s}(X,\cF(-r)) ;
		\end{align*}
		
		\item \label{rem:motivic_homology_functoriality:base_change}
		when $b_1 = \id_{Y_1}$, the \emph{base change} morphisms
		\begin{align*}
		h^* &\colon \CBM_\bullet(X_1/Y_1,\cF) \longrightarrow \CBM_\bullet(X_2/Y_1,\cF) \simeq h_* \CBM_\bullet(X_2/Y_2,h^\ast(\cF)), \\
		h^* &\colon \HBM_s(X_1/Y_1,\cF(r)) \longrightarrow \HBM_s(X_2/Y_2, h^\ast(\cF(r))) .
		\end{align*}
	\end{enumerate}
\end{remark}

\bigskip
Assume from now on that the coefficient $\cF$ admits an $\mathbb E_\infty$-ring structure, and denote by
\[ m \colon \cF \otimes \cF \longrightarrow \cF \]
the underlying multiplication.
This will induce various product operations on motivic (co)chains and (co)homologies as follows.

\begin{definition}
	Let $f \colon X \to Y$ be a morphism of derived \kanal stacks, $\cM \in \RigSHet(X,\bbQ)$ and $\cN \in \RigSHet(Y,\bbQ)$.
	We have a natural transformation
	\[ f_!(f^!(\cM) \otimes f^*(\cN)) \simeq f_!f^!(\cM) \otimes \cN \xrightarrow{\varepsilon_!} \cM \otimes \cN , \]
	where $\varepsilon_!$ denotes the counit of the adjunction $f_! \dashv f^!$, and the equivalence is given by the projection formula.
	Using the same adjunction again, we obtain an exchange transformation
	\[ \mathrm{Ex}^{!*}_\otimes(f) \colon f^!(\cM) \otimes f^\ast(\cN) \longrightarrow f^!(\cM \otimes \cN) .  \]
\end{definition}

\begin{definition}
	Let
	\[ \begin{tikzcd}[column sep = small]
		X \arrow{rr}{f} \arrow{dr}[swap]{a} & & Y \arrow{dl}{b} \\
		& S
	\end{tikzcd} \]
	be a commutative triangle of derived \kanal stacks.
	\begin{enumerate}[wide=0pt, itemsep=1ex]
		\item The \emph{left composition product} is the map
		\begin{equation}\label{eq:left_composition}
			\circ \colon \rC^\bullet(X/S,\cF) \otimes \CBM_\bullet(f/S,\cF) \longrightarrow \CBM_\bullet(f/S,\cF)
		\end{equation}
		that corresponds under the adjunctions $a^* \dashv a_*$ and $f_! \dashv f_!$ to the composition of
		\begin{multline*}
			f_!( a^* a_* a^*(\cF) \otimes a^* a_* f^! b^*(\cF) ) \xrightarrow{f_!(\varepsilon^*_a \otimes \varepsilon^*_a)} f_!( f^* b^*(\cF) \otimes f^! b^*(\cF) ) \\
			\xrightarrow{\mathsf{Ex}^{!*}_\otimes(f)} f_!f^!(b^*(\cF) \otimes b^*(\cF)) \xrightarrow{\varepsilon^!_f} b^*(\cF \otimes \cF) \xrightarrow{b^*(m)} b^*(\cF) ,
		\end{multline*}
		where $\varepsilon^*_a$ and $\varepsilon^!_f$ denote respectively the counits of the adjunctions $a^* \dashv a_*$ and $f_! \dashv f^!$.
		
		\item The \emph{right composition product} is the map
		\begin{equation}\label{eq:right_composition}
			\circ \colon \CBM_\bullet(f/S,\cF) \otimes \CBM_\bullet(Y/S,\cF) \longrightarrow \CBM_\bullet(X/S,\cF)
		\end{equation}
		that corresponds under the adjunctions $a^* \dashv a_*$ and $a_! \dashv a^!$ to the composition of
		\begin{multline*}
			a_!( a^* a_* f^! b^*(\cF) \otimes a^* b_* b^!(\cF) )  \xrightarrow{a_!(\varepsilon^*_a \otimes \varepsilon^*_b)} a_!( f^! b^*(\cF) \otimes f^* b^!(\cF) ) \xrightarrow{\mathrm{Ex}^{!*}_\otimes(f)} a_!f^!( b^*(\cF) \otimes b^!(\cF)) \\
			\xrightarrow{\mathrm{Ex}^{!*}_\otimes(b)} a_! f^! b^!( \cF \otimes \cF ) \simeq a_! a^!(\cF \otimes \cF) \xrightarrow{\varepsilon^!_a} \cF \otimes \cF \xrightarrow{m} \cF ,
		\end{multline*}
		where $\varepsilon^*_a$, $\varepsilon^*_b$ and $\varepsilon^!_a$ denote respectively the counits of the adjunctions $a^* \dashv a_*$, $b^* \dashv b_*$ and  $a_! \dashv a^!$.
	\end{enumerate}
\end{definition}

\begin{remark}
	(1) In the special case where $f = \id_X$, the left composition product yields the \emph{cup product} on motivic cochains
	\[ \cup \colon \rC^\bullet(X/S,\cF) \otimes \rC^\bullet(X/S,\cF) \longrightarrow \rC^\bullet(X/S,\cF) ; \]
	the right composition morphism yields the \emph{cap product} on motivic (co)chains
	\[ \cap \colon \rC^\bullet(X/S,\cF) \otimes \CBM_\bullet(X/S,\cF) \longrightarrow \CBM_\bullet(X/S,\cF).\]
	Note that the cap product can also be obtained as a special case of the left composition product for $f = a$ and $b = \id_S$.
	
	\smallskip
	(2) Fix integers $r_1, r_2, s_1, s_2 \in \bbZ$, $r \coloneqq r_1 + r_2$ and $s \coloneqq s_1 + s_2$.
	Tensoring \eqref{eq:left_composition} with $1_S(r_1)[s_1] \otimes 1_S(r_2)[s_2]\simeq 1_S(r)[s]$, we obtain a map
	\[ \circ \colon \rC^\bullet(X/S)(r_1)[s_1] \otimes \CBM_\bullet(f/S)(r_2)[s_2] \longrightarrow \CBM_\bullet(f/S)(r)[s] . \]
	Taking $f = \id_X$, and applying the lax-monoidal functor $\pi_0\Map_{\RigSHet(S,\bbQ)}(1_S,-)$, we obtain the \emph{cup product} on motivic cohomologies
	\[ \cup \colon \rH^{s_1}(X,\cF(r_1)) \otimes \rH^{s_2}(X,\cF(r_2)) \longrightarrow \rH^s(X,\cF(r)) . \]
	On the other hand, tensoring \eqref{eq:right_composition} with $1_S(-r)[-s]$, we obtain a map
	\[ \circ \colon \CBM_\bullet(f/S,\cF)(-r_1)[-s_1] \otimes \CBM_\bullet(Y/S,\cF)(-r_2)[-s_2] \longrightarrow \CBM_\bullet(X/S,\cF)(-r)[-s] . \]
	Applying $\pi_0\Map_{\RigSHet(S,\bbQ)}(1_S,-)$, we obtain the \emph{right composition product} on motivic homologies
	\[ \circ \colon \HBM_{s_1}(X/Y,b^*(\cF)(r_1)) \otimes \HBM_{s_2}(Y/S,\cF(r_2)) \longrightarrow \HBM_{s}(X/S,\cF(r)) . \]
	Taking furthermore $f = \id_X$, we obtain the \emph{cap product} on motivic (co)homologies
	\[ \cap \colon \rH^{-s_1}(X,\cF(-r_1)) \otimes \HBM_{s_2}(X,\cF(r_2)) \longrightarrow \HBM_s(X/S,\cF(r)) . \]
\end{remark}

\begin{definition} \label{def:external_product}
	 Let
	\[ \begin{tikzcd}
		X \times_S Y \rar{p_1} \dar{p_2} \arrow{dr}{p} & Y \dar{b} \\
		X \rar{a} & S
	\end{tikzcd} \]
	be a pullback square of derived \kanal stacks.
	\begin{enumerate}[wide=0pt, itemsep=1ex]
		\item The \emph{external product} of motivic cochains
		\[ \boxtimes \colon \rC^\bullet(X/S,\cF) \otimes \rC^\bullet(Y/S,\cF) \longrightarrow \rC^\bullet(X\times_S Y / S,\cF) \]
		is defined as the composition of
		\begin{multline*}
			\rC^\bullet(X/S,\cF) \otimes \rC^\bullet(Y/S,\cF)  \xrightarrow{p_1^\ast \otimes p_2^\ast} \rC^\bullet(X \times_S Y/S,\cF) \otimes \rC^\bullet(X \times_S Y/S,\cF) \\
			\xrightarrow{\cup} \rC^\bullet(X\times_S Y/S,\cF) .
		\end{multline*}
		It induces the \emph{external product} of motivic cohomology
		\[ \boxtimes \colon \rH^{s_1}(X/S,\cF(r_1)) \otimes \rH^{s_2}(Y/S,\cF(r_2)) \longrightarrow \rH^{s_1 + s_2}(X \times_S Y / S, \cF(r_1 + r_2))\]
		for $r_1, r_2, s_1, s_2 \in \bbZ$.
	
		\item The \emph{external product} of motivic Borel-Moore chains is the map
		\[ \boxtimes \colon \CBM_\bullet(X/S,\cF) \otimes \CBM_\bullet(Y/S,\cF) \longrightarrow \CBM_\bullet(X \times_S Y / S, \cF) \]
		that corresponds under the adjunctions $p^* \dashv p_*$ and $p_! \dashv p^!$ to the composition of
		\begin{multline*}
			p_!( p^* a_* a^!(\cF) \otimes p^* b_* b^!(\cF) ) \to a_! p_{2,!}( p_2^* a^!(\cF) \otimes p_1^* b^!(\cF) ) \simeq a_!(a^!(\cF) \otimes a^* b_! b^!(\cF)) \\
			\to a_!(a^!(\cF) \otimes a^*(\cF)) \to a_! a^!(\cF) \to \cF .
		\end{multline*}
		It induces the \emph{external product} of motivic Borel-Moore homology
		\[ \boxtimes \colon \HBM_{s_1}(X/S,\cF(r_1)) \otimes \HBM_{s_2}(Y/S,\cF(r_2)) \longrightarrow \HBM_{s_1 + s_2}(X \times_S Y / S, \cF(r_1 + r_2)) \]
		for $r_1, r_2, s_1, s_2 \in \bbZ$.
	\end{enumerate}
\end{definition}

We refer to \cite{Mazza_Lecture_notes_on_motivic_cohomology,Deglise_Bivariant_theories,Khan_Virtual_fundamental_classes} for the basic properties of motivic Borel-Moore homologies, including base change formulas, functoriality of the composition product and projection formulas for the composition product (see in particular \cite[\S 2.3]{Khan_Virtual_fundamental_classes}).

The following Propositions \ref{prop:pullback_cap}--\ref{prop:external_product_pushforward} encode further compatibility relations between pullback, pushforward, composition product, cap product, external product of motivic cohomology and motivic Borel-Moore homology at the chain level.
All the proofs are similar, and rely on long but straightforward diagram chasing.
We will sketch the first proof and omit the rest.

\begin{proposition}\label{prop:pullback_cap}
	Let $S$ be a derived \kanal stack and let $f \colon X \to Y$ be a morphism of derived \kanal stacks over $S$.
	Then the diagram
	\[ \begin{tikzcd}
		\rC^\bullet(Y/S,\cF) \otimes \CBM_\bullet(f/S,\cF) \otimes \CBM_\bullet(Y/S,\cF) \rar{f^* \otimes \circ} \dar{\simeq} & \rC^\bullet(X/S,\cF) \otimes \CBM_\bullet(X/S,\cF) \arrow{dd}{\cap} \\
		\CBM_\bullet(f/S,\cF) \otimes \rC^\bullet(Y/S,\cF) \otimes \CBM_\bullet(Y/S,\cF) \dar{\id \otimes \cap} \\
		\CBM_\bullet(f/S,\cF) \otimes \CBM_\bullet(Y/S,\cF) \rar{\circ} & \CBM_\bullet(X/S,\cF)
	\end{tikzcd} \]
	is canonically commutative.
	\end{proposition}

\begin{proof}
	Expanding the definitions, we need to show that the square
	\[ \begin{tikzcd}
		b_* b^* \cF \otimes a_* f^! b^* \cF \otimes b_* b^! \cF \rar \dar & a_* a^* \cF \otimes a_* a^! \cF \dar \\
		a_* f^! b^* \cF \otimes b_* b^! \cF \rar & a_* a^! \cF
	\end{tikzcd} \]
	is commutative.
	Using the adjunctions $a^* \dashv a_*$ and $a_! \dashv a^!$, this is equivalent to the commutativity of the following diagram
	\[ \begin{tikzcd}
		a_!( a^* b_* b^* \cF \otimes a^* a_* f^! b^* \cF \otimes a^* b_* b^! \cF) \rar \dar & a_!( a^* a_* a^* \cF \otimes a^* a_* a^! \cF ) \dar \\
		a_!(a^* a_* f^! b^* \cF \otimes a^* b_* b^! \cF ) \rar & \cF .
	\end{tikzcd} \]
	A long but straightforward diagram chase reduces this question to the commutativity of the diagram
	\[ \begin{tikzcd}
		\cF \otimes \cF \otimes \cF \rar{\id \otimes m} \dar{m \otimes \id} & \cF \otimes \cF \dar{m} \\
		\cF \otimes \cF \rar{m} & \cF ,
	\end{tikzcd} \]
	which is indeed part of the given $\mathbb E_\infty$-structure of $\cF$.
\end{proof}

\begin{proposition} \label{prop:projection_formula_cap}
	Let $S$ be a derived \kanal stack and let $f \colon X \to Y$ be a proper morphism of derived \kanal stacks over $S$.
	Then the diagram
	\[ \begin{tikzcd}
		\rC^\bullet(Y/S,\cF) \otimes \CBM_\bullet(X/S,\cF) \rar{f^* \otimes \id} \arrow{dd}{\id \otimes f_*} & \rC^\bullet(X/S,\cF) \otimes \CBM_\bullet(X/S,\cF) \dar{\cap} \\
		& \CBM_\bullet(X/S,\cF) \dar{f_*} \\
		\rC^\bullet(Y/S,\cF) \otimes \CBM_\bullet(Y/S,\cF) \rar{\cap} & \CBM_\bullet(Y/S,\cF)
	\end{tikzcd} \]
	is canonically commutative.
	\end{proposition}

\begin{proposition}\label{prop:projection_formula_Gysin}
	Let $S$ be a derived \kanal stack and let $f \colon X \to Y$ be a proper morphism of derived \kanal stacks over $S$.
	Then the diagram
	\[ \begin{tikzcd}
		\rC^\bullet(X/S,\cF) \otimes \CBM_\bullet(f/S,\cF) \otimes \CBM_\bullet(Y/S,\cF) \rar{\id \otimes \circ} \dar{\circ \otimes \id} & \rC^\bullet(X/S,\cF) \otimes \CBM_\bullet(X/S,\cF) \dar{\cap} \\
		\CBM_\bullet(f/S,\cF) \otimes \CBM_\bullet(Y/S) \dar{f_* \otimes \id} & \CBM_\bullet(X/S) \dar{f_*} \\
		\rC^\bullet(Y/S,\cF) \otimes \CBM_\bullet(Y/S,\cF) \rar{\cap} & \CBM_\bullet(Y/S,\cF)
	\end{tikzcd} \]
	is canonically commutative.
	\end{proposition}

\begin{proposition}\label{prop:external_product_pullback}
	Let
	\[ \begin{tikzcd}
		X \times_S Y \rar{p_1} \dar{p_2} \arrow{dr}{p} & Y \dar{b} \\
		X \rar{a} & S
	\end{tikzcd} \]
	be a pullback square of derived \kanal stacks.
	Then the diagrams
	\[ \begin{tikzcd}
		\CBM_\bullet(X/S,\cF) \otimes \CBM_\bullet(Y/S,\cF) \rar{p_2^* \otimes \id} \dar{\id \otimes p_1^*} \arrow{dr}{\boxtimes} & \CBM_\bullet(p_1/S,\cF) \otimes \CBM_\bullet(Y/S,\cF) \dar{\circ} \\
		\CBM_\bullet(X/S,\cF) \otimes \CBM_\bullet(p_2/S,\cF) \rar{\circ} & \CBM_\bullet(X\times_S Y/S,\cF)
	\end{tikzcd} \]
	and
	\[ \begin{tikzcd}
		\rC^\bullet(X/S,\cF) \otimes \CBM_\bullet(X/S,\cF) \otimes \rC^\bullet(Y/S,\cF) \otimes \CBM_\bullet(Y/S,\cF) \dar{\simeq} \rar{\cap \otimes \cap} & \CBM_\bullet(X/S,\cF) \otimes \CBM_\bullet(Y/S,\cF) \arrow{dd}{\boxtimes} \\
		\rC^\bullet(X/S,\cF) \otimes \rC^\bullet(Y/S,\cF) \otimes \CBM_\bullet(X/S,\cF) \otimes \CBM_\bullet(Y/S,\cF) \dar{\boxtimes \otimes \boxtimes} \\
		\rC^\bullet(X \times_S Y/S,\cF) \otimes \CBM_\bullet(X \times_S Y/S,\cF) \rar{\cap} & \CBM_\bullet(X \times_S Y/S,\cF)
	\end{tikzcd} \]
	are canonically commutative.
	\end{proposition}

\begin{proposition} \label{prop:external_product_pushforward}
	Let
	\[ \begin{tikzcd}
		X'\times_S Y' \arrow{dr}{h} \arrow{rr} \arrow{dd} & & Y' \dar{g} \\
		& X \times_S Y \rar{p_1} \dar{p_2} \arrow{dr}{p} & Y \dar{b} \\
		X' \rar{f} & X \rar{a} & S
	\end{tikzcd} \]
	be a commutative diagram of derived \kanal stacks where the two squares are cartesian, and the morphisms $f$, $g$, $h$ are proper.
	Then the diagram
	\[ \begin{tikzcd}
		\CBM_\bullet(X'/S,\cF) \otimes \CBM_\bullet(Y'/S,\cF) \rar{f_* \otimes g_*} \dar{\boxtimes} & \CBM_\bullet(X/S,\cF) \otimes \CBM_\bullet(Y/S,\cF) \dar{\boxtimes} \\
		\CBM_\bullet(X' \times_S Y'/S, \cF) \rar{h_*} & \CBM_\bullet(X \times_S Y / S, \cF )
	\end{tikzcd} \]
	is canonically commutative.
\end{proposition}

\section{Virtual fundamental classes of derived stacks} \label{sec:vfc}

In this section, we sketch the general construction of virtual fundamental classes for lci derived non-archimedean analytic stacks, which is parallel to the construction in the algebraic case by Khan \cite{Khan_Virtual_fundamental_classes}.
The idea is to adapt the Gysin map in classical intersection theory to the derived setting, via deformation to the shifted tangent bundle by Khan and Rydh \cite{Khan_Virtual_Cartier_divisors}.

Let us first recall the notion of derived lci from \cite[\S 2]{Porta_Yu_Non-archimedean_quantum_K-invariants}.

\begin{definition}
	A morphism $f\colon X\to Y$ of derived \kanal stacks is called \emph{derived lci} if its analytic cotangent complex $\anL_{X/Y}$ is perfect and in tor-amplitude $(-\infty,1]$.
	The \emph{relative virtual dimension} at any geometric point $x\in\trunc(X)$ is the Euler characteristic of the pullback of $\anL_{X/Y}$.
	It is locally constant by \cite[Lemma 2.11]{Porta_Yu_Non-archimedean_quantum_K-invariants}.
\end{definition}

We will sometimes say \emph{lci} for short because derived lci is equivalent to lci in the underived case by \cite[Lemma 2.4]{Porta_Yu_Non-archimedean_quantum_K-invariants}.
Derived lci is also commonly referred to as quasi-smooth, but we avoid this terminology due to possible confusion with other notions in Berkovich geometry.

\begin{definition} \label{def:normal_bundle_stack}
	For any lci morphism $X\to Y$ of derived \kanal stacks, the \emph{1-shifted tangent bundle} is the vector bundle stack
	\[\bbT_{X/Y}[1]\coloneqq\bV_X(\anL_{X/Y}[-1])\]
	over $X$.
\end{definition}

There is a canonical deformation from $X\to Y$ to the 0-section $X\to \bbT_{X/Y}[1]$.

\begin{theorem}[{\cite[Theorem 1.3]{Khan_Virtual_fundamental_classes}}] \label{thm:deformation_to_tangent}
	There exists an lci derived \kanal stack $D_{X/Y}$ over $Y\times\bA^1$, and an lci morphism
	\[X\times\bA^1\longto D_{X/Y}\]
	over $Y\times\bA^1$.
	The fiber over $\bGm=\bA^1\setminus 0$ is $X\times\bGm\to Y\times\bGm$, and the fiber over $0\in\bA^1$ is the 0-section $X\to\bbT_{X/Y}[1]$.
\end{theorem}

Now fix a derived \kanal stack $S$, a coefficient $\cF\in\RigSHet(S,\bbQ)$, and $f\colon X\to Y$ an lci morphism of derived \kanal stacks over $S$ of relative virtual dimension $d$.

The deformation to the 1-shifted tangent bundle in \cref{thm:deformation_to_tangent} gives rise to a \emph{specialization map}
\[\rsp_{X/Y}\colon\HBM_s(Y/S,\cF(r))\longto\HBM_s(\bbT_{X/Y}[1]/S,\cF(r))\quad\text{for every }r,s\in\bbZ.\]
By the homotopy invariance for vector bundles, the target is identified with $\HBM_{s+2d}(X/S,\cF(r+d))$.
So we obtain a canonical map
\[\varphi\colon\HBM_s(Y/S,\cF(r))\longto\HBM_{s+2d}(X/S,\cF(r+d)).\]

\begin{definition} \label{def:virtual_fundamental_class}
	The \emph{virtual fundamental class} of $f\colon X\to Y$ is the class
	\[[X/Y]\coloneqq \varphi(1)\in\HBM_{2d}(X/Y,\cF(d))\]
	where $1\in\HBM_0(Y/Y,\cF)$, and $d$ is the relative virtual dimension of $f$.
\end{definition}

\begin{definition} \label{def:Gysin_pullback}
	The \emph{Gysin pullback} is the morphism
	\[ f^! \colon \CBM_\bullet(Y/S,\cF) \longrightarrow \CBM_\bullet(X/S,\cF(-d))[-2d] \]
	given by the right composition product with $[X/Y]$ on the left.
	For $r,s \in \bbZ$, this induces the \emph{Gysin pullback} for motivic Borel-Moore homology
	\[ f^! \colon \HBM_s(Y/S,\cF(r)) \longrightarrow \HBM_{s+2d}(X/S,\cF(r+d)) . \]
\end{definition}

We refer to \cite[\S 3]{Khan_Virtual_fundamental_classes} for more details on virtual fundamental classes.
Two important properties of the virtual fundamental class construction (or equivalently of the Gysin map) are functoriality and base change, which we describe in the following two propositions.
They are analogous to \cite[Propositions 7.2 and 7.5]{Behrend_Intrinsic_normal_cone} in the classical approach via perfect obstruction theory.

\begin{proposition}[Functoriality] \label{prop:vfc_functoriality}
	Let $f\colon X\to Y$ and $g\colon Y\to Z$ be lci morphisms of derived \kanal stacks of relative virtual dimensions $d$ and $e$ respectively.
	We have $f^! \circ g^! = (g\circ f)^!$, and
	\[[X/Y]\circ[Y/Z]=[X/Z] \in \HBM_{2d+2e}(X/Z,\cF(d+e)).\]
\end{proposition}

\begin{proposition}[Base change] \label{prop:vfc_base_change}
	Let
	\[\begin{tikzcd}
		X' \rar \dar & Y' \dar{q} \\
		X \rar{f} & Y
	\end{tikzcd}\]
	be a cartesian square of derived \kanal stacks, where $f$ is derived lci of relative virtual dimension $d$.
	We have
	\[q^*[X/Y]=[X'/Y']\in\HBM_{2d}(X'/Y',\cF(d)) , \]
	where $q^*$ denotes the base change morphism of \cref{rem:motivic_homology_functoriality}(\ref{rem:motivic_homology_functoriality:base_change}).
\end{proposition}

\begin{remark} \label{rem:Gysin}
Given a pullback square of derived \kanal stacks over $S$,
\[\begin{tikzcd}
X'\rar{w} \dar & Y'\dar\\
X\rar{v} & Y
\end{tikzcd}\]
where $v$ and $w$ are derived lci of virtual dimension $d$, it is convenient to denote
\[v^!\coloneqq w^!\colon\HBM_s(Y'/S,\cF(r))\longto\HBM_{s+2d}(X'/S,\cF(r+d))\quad\text{for every }r,s\in\bbZ.\]
This notation is justified by the base change property of deformation to the shifted tangent bundle (\cite[\S 1.4]{Khan_Virtual_fundamental_classes}).
When $X,Y,S$ are all underived, the functor $v^!$ is thus compatible with the classical refined Gysin homomorphism (see \cite[\S 6.2]{Fulton_Intersection_theory}, \cite[\S 7]{Behrend_Intrinsic_normal_cone}), in particular it is independent of the derived structures on $X'$ and $Y'$.
\end{remark}

\begin{definition} \label{def:Gysin_pushforward}
	Let $f \colon X \to Y$ be a morphism of derived \kanal stacks over $S$.
	Assume that $f$ is proper and derived lci of relative dimension $d$.
	We define the \emph{Gysin pushforward}
	\[ f_! \colon \rC^\bullet(X/S,\cF) \longrightarrow \rC^\bullet(Y/S,\cF(-d))[-2d] \]
	as the composition of
	\[ \rC^\bullet(X/S,\cF) \xrightarrow{\circ [X/Y]} \CBM_\bullet(f/S,\cF(-d))[-2d] \xrightarrow{f_*} \rC^\bullet(Y/S,\cF(-d))[-2d], \]
	where $f_*$ is the proper pushforward in \cref{rem:motivic_homology_functoriality}(\ref{rem:motivic_homology_functoriality:pushforward_homology_to_cohomology}).
	For $r,s\in\bbZ$, this induces the \emph{Gysin pushforward} for motivic cohomology
	\[ f_! \colon \rH^s(X/S,\cF(r)) \longrightarrow \rH^{s-2d}(Y/S,\cF(r-d)) . \]
\end{definition}

\bigskip
The proposition below collects the properties of various product operations on motivic (co)homology with respect to pullbacks and pushforwards.

\begin{proposition} \label{prop:properties_of_products}
	\begin{enumerate}[wide=0pt, itemsep=1ex]
		\item (Compatibility between pullback and cap product)
		\label{prop:properties_of_products:pullback_cap} Let $S$ be a derived \kanal stack and $f\colon X\to Y$ an lci morphism of derived \kanal $S$-stacks of relative dimension $d$.
		For any classes $\alpha\in\rH^s(Y,\cF(r))$, $\beta\in \HBM_{s'}(Y/S,\cF(r'))$, we have
		\[f^*(\alpha)\cap f^!(\beta) = f^!(\alpha\cap\beta)\]
		in $\HBM_{s'-s+2d}(X/S,\cF(r'-r+d))$.
		
		\item (Projection formula for cap product)
		\label{prop:properties_of_products:projection_formula_cap}
		Same context as above, for any classes $\alpha\in\rH^s(Y,\cF(r))$, $\beta\in \HBM_{s'}(Y/S,\cF(r'))$, we have
		\[f_*(f^*(\alpha)\cap\beta) = \alpha\cap f_*(\beta)\]
		in $\HBM_{s'-s}(X/S,\cF(r'-r))$.
		
		\item (Projection formula for cap product with Gysin)
		\label{prop:properties_of_products:projection_formula_cap_Gysin}
		Same context as above, for any classes $\alpha\in\rH^s(Y,\cF(r))$, $\beta\in \HBM_{s'}(Y/S,\cF(r'))$, we have
		\[f_*(\alpha\cap f^!(\beta)) = f_!(\alpha)\cap\beta\]
		in $\HBM_{s'-s+2d}(X/S,\cF(r'-r+d))$.

		\item (Comparison between external product and composition product with base change)
		\label{prop:properties_of_products:external_and_composition}
		Let
		\[ \begin{tikzcd}
			X \times_S Y \rar{p_1} \dar{p_2} \arrow{dr}{p} & Y \dar{b} \\
			X \rar{a} & S
		\end{tikzcd} \]
		be a pullback square of derived \kanal stacks.
		For any classes $\alpha\in\HBM_{s_1}(X/S,\cF(r_1))$, $\beta\in \HBM_{s_2}(Y/S,\cF(r_2))$, we have
		\[a^*(\beta)\circ\alpha = \alpha\boxtimes\beta = b^*(\alpha)\circ\beta\]
		in $\HBM_{s_1+s_2}(X\times_S Y/S,\cF(r_1+r_2))$.
		
		\item (Compatibility between external product and cap product)
		\label{prop:properties_of_products:compatibility_external_cap}
		Same context as above, for any classes $\alpha\in\rH^{s_1}(X,\cF(r_1))$, $\alpha'\in\HBM_{s'_1}(X/S,\cF(r'_1))$, $\beta\in \rH^{s_2}(Y,\cF(r_2))$, $\beta'\in \HBM_{s'_2}(Y/S,\cF(r'_2))$, we have
		\[(\alpha\boxtimes\beta)\cap(\alpha'\boxtimes\beta')=(\alpha\cap\alpha')\boxtimes(\beta\cap\beta')\]
		in $\HBM_{s'_1-s_1+s'_2-s_2}(X\times_S Y/S,\cF(r'_1-r_1+r'_2-r_2))$.
		
		\item (Compatibility between external product and pushforward)
		\label{prop:properties_of_product:external_product_pushforward}
		Let
		\[ \begin{tikzcd}
			X'\times_S Y' \arrow{dr}{h} \arrow{rr} \arrow{dd} & & Y' \dar{g} \\
			& X \times_S Y \rar{p_1} \dar{p_2} \arrow{dr}{p} & Y \dar{b} \\
			X' \rar{f} & X \rar{a} & S
		\end{tikzcd} \]
		be a commutative diagram of derived \kanal stacks where the two squares are cartesian, and the morphisms $f$, $g$, $h$ are proper.
		For any classes $\alpha\in\HBM_{s_1}(X'/S,\cF(r_1))$, $\beta\in \HBM_{s_2}(Y'/S,\cF(r_2))$, we have
		\[ h_*(\alpha\boxtimes\beta) = f_*(\alpha)\boxtimes g_*(\beta) \]
		in $\HBM_{s_1+s_2}(X\times_S Y/S,\cF(r_1+r_2))$.
	\end{enumerate}
\end{proposition}
\begin{proof}
	The statements follow from Propositions \ref{prop:pullback_cap}--\ref{prop:external_product_pushforward} respectively.
\end{proof}

Next we state two base change formulas, the first relates star pushforward with Gysin pullback for motivic Borel-Moore homology, and the second relates Gysin pushforward with star pullback for motivic cohomology.

\begin{proposition} \label{prop:Gysin_pullback_base_change}
	Let $S$ be a derived \kanal stack, and
	\[\begin{tikzcd}
		X' \rar{g} \dar{p} & Y' \dar{q} \\
		X \rar{f} & Y
	\end{tikzcd}\]
	a pullback square of derived \kanal stacks, where $f$ is derived lci of relative dimension $d$.
	For any class $\alpha\in\HBM_s(Y'/S,\cF(r))$, we have
	\[f^!q_*(\alpha) = p_*g^!(\alpha)\]
	in $\HBM_{s+2d}(X/S,\cF(r+d))$.
\end{proposition}

\begin{proof}
	We have
	\begin{align*}
		f^! q_*(\alpha) & = [X/Y] \circ q_*(\alpha) && \text{by \cref{def:Gysin_pullback}} \\
		& = p_*( q^*[X/Y] \circ \alpha ) && \text{by the projection formula (\cite[\S 2.3.4]{Khan_Virtual_fundamental_classes})} \\
		& = p_* ([X'/Y'] \circ \alpha)  && \text{by \cref{prop:vfc_base_change}} \\
		& = p_* g^!(\alpha) && \text{by  \cref{def:Gysin_pullback}.}
	\end{align*}
\end{proof}

\begin{proposition} \label{prop:Gysin_pushforward_base_change}
	Let $S$ be a derived \kanal stack, and
	\[\begin{tikzcd}
		X' \rar{g} \dar{p} & Y' \dar{q} \\
		X \rar{f} & Y
	\end{tikzcd}\]
	a pullback square of derived \kanal stacks, where $f$ is proper and derived lci of relative dimension $d$.
	For any class $\alpha\in\rH^s(X/S,\cF(r))$, we have
	\[q^*f_!(\alpha)=g_!p^*(\alpha)\]
	in $\rH^{s-2d}(Y'/S,\cF(r-d))$.
\end{proposition}

\begin{proof}
	We have
	\begin{align*}
		q^*f_!(\alpha) & = q^*(f_*(\alpha \circ [X/Y])) && \text{by \cref{def:Gysin_pushforward}} \\
		& = g_* p^*(\alpha \circ [X/Y]) && \text{by the base change formula (\cite[\S 2.3.2]{Khan_Virtual_fundamental_classes})} \\
		& = g_* (p^*(\alpha) \circ q^*[X/Y]) && \text{by \cite[\S 2.3.1]{Khan_Virtual_fundamental_classes}} \\
		& = g_*( p^*(\alpha) \circ [X'/Y'] ) && \text{by \cref{prop:vfc_base_change}} \\
		& = g_! p^*(\alpha)  && \text{by \cref{def:Gysin_pushforward}.}
	\end{align*}
\end{proof}

\section{Non-archimedean Gromov-Witten virtual fundamental classes}

In this section, we construct the virtual fundamental classes for non-archimedean Gromov-Witten theory, using the general machinery in the previous sections.
We first give a review of derived moduli stacks of non-archimedean stable maps following \cite{Porta_Yu_Non-archimedean_quantum_K-invariants}.

Fix $S$ a rigid \kanal space and $X$ a rigid \kanal space smooth over $S$.
For any derived \kanal space $T$ over $S$, an $n$-pointed genus $g$ \emph{stable map} into $X/S$ over $T$ consists of a proper flat morphism $p\colon C\to T$, $n$ distinct sections $s_i\colon T\to C$ and an $S$-map $f\colon C\to X$, such that every geometric fiber $[C_t,(s_i(t)),f_t\colon C_t\to X]$ is a stable map, in the sense that its automorphism group is a finite \kanal group\footnote{It is a finite constant group when the ground field $k$ has characteristic zero.}.

Let $\RoM_{g,n}(X/S)$ denote the derived moduli stack of $n$-pointed genus $g$ stable maps into $X/S$.
Using the representability theorem in derived non-archimedean geometry (see \cite{Porta_Yu_Representability_theorem}) and an explicit computation of the analytic cotangent complex, we have the following.

\begin{theorem}[{\cite[\S 4.2]{Porta_Yu_Non-archimedean_quantum_K-invariants}}] \label{thm:derived_moduli_stack}
	The derived stack $\RoM_{g,n}(X/S)$ is an lci derived \kanal stack locally of finite presentation over $S$.
\end{theorem}

Apply \cref{def:virtual_fundamental_class} to the lci derived \kanal stack $\RoM_{g,n}(X/S)$ and the coefficient $\cF=\bbQ_S$, we obtain the associated virtual fundamental class \[[\RoM_{g,n}(X/S)]\in\HBM_*\big(\RoM_{g,n}(X/S)/S,\bbQ_S(*)\big)\]
in the motivic Borel-Moore homology of the moduli stack.

In order to establish all the expected properties of non-archimedean Gromov-Witten invariants, instead of working with $n$-pointed genus $g$ stable maps, we need use a slight combinatorial refinement called $(\tau,\beta)$-marked stable maps associated to an A-graph $(\tau,\beta)$, introduced by Behrend-Manin \cite{Behrend_Stacks_of_stable_maps}.

Let us recall the basic definitions.
A \emph{modular graph} consists of the following data:
\begin{enumerate}
	\item A finite graph $\tau$.
	\item A subset $T_\tau$ of 1-valent vertices of $\tau$ called \emph{tail vertices} and an ordering on $T_\tau$.
	We require that each edge of $\tau$ contains at most one tail vertex.
	We denote by $V_\tau$ the set of non-tail vertices.
	For every $v\in V_\tau$, we denote by $E_v$ the set of edges connected to $v$, and by $\val(v)$ the cardinality of $E_v$.
	\item For each $v\in V_\tau$, a natural number $g(v)$ called \emph{genus}.
\end{enumerate}
For the purpose of this paper, we will restrict to stable modular graphs, i.e. $2g(v) + \val(v) \ge 3$ for every vertex v.
An A-graph $(\tau,\beta)$ consists of a modular graph $\tau$ and a map $\beta\colon V_\tau\to A(X)$, the group of one-dimensional cycles in $X$ modulo analytic equivalence.

Given a modular graph $\tau$ and a derived \kanal space $T$, a $\tau$-marked prestable curve over $T$ consists for every $v\in V_\tau$, a $\val(v)$-pointed genus $g(v)$ prestable curve $[C_v,(s_{v,e})_{e\in E_v}]$ over $T$.
For each edge $e$ of $\tau$ not containing a tail vertex, we can glue the two corresponding sections $s_{v,e}$.
In this way, we obtain a \emph{glued curve} from a $\tau$-marked prestable curve.

Given an A-graph $(\tau,\beta)$ and a derived \kanal space $T$, a $(\tau,\beta)$-marked prestable map into $X/S$ over $T$ consists of a $\tau$-marked prestable curve $[C_v,(s_{v,e})_{e\in E_v}]_{v\in V_\tau}$ over $T$ and an $S$-map from the associated glued curve $[C,(s_i)_{i\in T_\tau},(s_e)_{e\in E_\tau}]$ to $X$, such that for every geometric point $t\in T$ and every $v\in V_\tau$, the composite map $C_{v,t}\to C_t\xrightarrow{f_t} X$ has class $\beta(v)$.
It is called \emph{stable} if every geometric fiber $[C_t,(s_i(t))_{i\in T_\tau},f_t]$ is a stable map.

\begin{theorem}[{\cite[\S 4-5]{Porta_Yu_Non-archimedean_quantum_K-invariants}}] \label{thm:geometric_relations}
	 The derived moduli stack $\RoM(X/S,\tau,\beta)$ of $(\tau,\beta)$-marked stable maps into $X/S$ is again an lci derived \kanal stack locally of finite presentation over $S$.
	 Furthermore, it satisfies the following geometric relations with respect to elementary operations on A-graphs:
	\begin{enumerate}[wide=0pt, itemsep=1ex]
	\item \label{thm:geometric_relations:products} Products: Let $(\tau_1, \beta_1)$ and $(\tau_2, \beta_2)$ be two A-graphs.
	We have a canonical equivalence
	\[ \RoM(X/S, \tau_1 \sqcup \tau_2, \beta_1 \sqcup \beta_2) \xrightarrow{\ \sim\ } \RoM(X/S, \tau_1, \beta_1) \times_S  \RoM(X/S, \tau_2, \beta_2). \]
	
	\item \label{thm:geometric_relations:cutting_edges} Cutting edges: Let $(\sigma,\beta)$ be an A-graph obtained from $(\tau,\beta)$ by cutting an edge $e$ of $\tau$.
	Let $v,w$ be the two tails of $\sigma$ created by the cut.
	We have a derived pullback diagram
	\[ \begin{tikzcd}
		\RoM(X/S, \tau, \beta) \rar{c} \dar{\ev_e} & \RoM(X/S, \sigma, \beta) \dar{\ev_v \times \ev_w} \\
		X \rar{\Delta} & X \times_S X.
	\end{tikzcd} \]
	
	\item \label{thm:geometric_relations:universal_curve} Universal curve: Let $(\sigma,\beta)$ be an A-graph obtained from $(\tau,\beta)$ by forgetting a tail attached to a vertex $w$.
	Let $\oCpre_w\to\oMpre_\sigma$ be the universal curve corresponding to $w$.
	We have a derived pullback diagram
	\[ \begin{tikzcd}
		\RoM( X/S, \tau, \beta ) \rar \dar & \RoM( X/S, \sigma, \beta ) \dar \\
		\oCpre_w \rar & \oMpre_\sigma.
	\end{tikzcd} \]
	
	\item \label{thm:geometric_relations:forgetting_tails} Forgetting tails:
	Same context as above, we have a derived pullback diagram
	\[ \begin{tikzcd}
		\RoM(X/S, \tau, \beta) \rar \dar & \oM_\tau \times_{\oM_\sigma}  \RoM(X/S, \sigma, \beta) \dar \\
		\oCpre_w \rar & \oM_\tau \times_{\oM_\sigma} \oMpre_\sigma.
	\end{tikzcd} \]
	
	\item \label{thm:geometric_relations:contracting_edges} Contracting edges:
	Let $(\sigma,\beta)$ be an A-graph where $\sigma$ is obtained from a modular graph $\tau$ by contracting an edge (possibly a loop) $e$.
	For each integer $l\ge 1$, we replace the edge $e$ in $\tau$ by a chain $c_l$ of $l$ edges where the new $(l-1)$ vertices have genus 0, and obtain a new modular graph $\tau_l$.
	Let $\beta^j_l, j\in J_l$ be all possible curve classes on $\tau_l$ that project to $\beta$ under the contraction of $c_l$.
	We have a canonical equivalence
	\[\colim_l \coprod_{j} \RoM(X/S,\tau_l,\beta^j_l)\xrightarrow{\ \sim\ }\oM_\tau\times_{\oM_\sigma}\RoM(X/S,\sigma,\beta).\]
\end{enumerate}
\end{theorem}

Apply \cref{def:virtual_fundamental_class} to the lci derived \kanal stack $\RoM(X/S,\tau,\beta)$ and the coefficient $\cF=\bbQ_S$, we obtain the associated virtual fundamental class \[[\RoM(X/S,\tau,\beta)]\in\HBM_d\big(\RoM(X/S,\tau,\beta)/S,\bbQ_S(2d)\big)\]
for $(\tau,\beta)$-marked stable maps as well, where $d$ denotes the virtual dimension.

\section{Proofs of Behrend-Manin axioms} \label{sec:axioms}

In this section, we prove a list of natural geometric properties for the virtual fundamental class $[\RoM(X/S,\tau,\beta)]$ of the moduli stack of $(\tau,\beta)$-marked stable maps with respect to elementary operations on graphs, namely, products, cutting edges, forgetting tails and contracting edges.
They are direct consequences of the corresponding geometric relations between the derived moduli stacks in \cref{thm:geometric_relations}.
Furthermore, we compute the virtual fundamental class in the case of mapping to a point via a deformation to the normal bundle.
All the above properties together constitute exactly the list of axioms for Gromov-Witten virtual fundamental classes proposed by Behrend and Manin in \cite[Definition 7.1]{Behrend_Stacks_of_stable_maps}.

\subsection{Mapping to a point} \label{sec:mapping_to_a_point}

\begin{theorem} \label{thm:mapping_to_a_point_vfc}
	Let $(\tau,0)$ be any A-graph where $\beta$ is 0.
	We have $\oM(X/S,\tau,0)\simeq\oM_\tau\times X$, and
	\[[\RoM(X/S,\tau,0)]=c_{g(\tau)\dim X/S}\big(\rR^1\pi_*\cO_{\oC_\tau}\boxtimes\bbT\an_{X/S}\big)\cap[\oM(X/S,\tau,0)],\]
	where $\pi\colon\oC_\tau\to\oM_\tau$ and $c_\cdot(\cdot)$ denotes the Chern class.
\end{theorem}

\begin{proof}
	The isomorphism $\oM(X/S,\tau,0)\simeq\oM_\tau\times X$ follows from \cite[Corollary 2.3]{Behrend_Stacks_of_stable_maps}.
	Let $\mathfrak M(X/S,\tau,0)$ be the $\bA^1_k$-deformation to the normal bundle associated to $j \colon \oM_\tau \times X \hookrightarrow \RoM(X/S,\tau,0)$ introduced in \cite[\S 6.4]{Porta_Yu_Non-archimedean_quantum_K-invariants}.
	By \cite[Proposition 6.13]{Porta_Yu_Non-archimedean_quantum_K-invariants},
	\[ \trunc( \mathfrak M(X/S, \tau, 0) ) \simeq \oM_\tau \times X \times \bA^1_k,\] 
	and we have derived fiber squares
	\[ \begin{tikzcd}
	\mathfrak N \rar \dar & \mathfrak M(X/S, \tau, 0) \dar & \RoM(X/S, \tau, 0) \dar \arrow{l} \\
	\Sp(k) \rar{0} & \bA^1_k & \Sp(k), \arrow{l}[swap]{1}
	\end{tikzcd} \]
	where 
	\[ \mathfrak N = \Spec_{\oM_\tau \times X}( \Sym_{\cO_{\oM_\tau \times X}}( \anL_j[-1])),\]
	and
	\[ \anL_j [-1] \simeq \big( \rR^1 \pi_* \cO_{\oC_\tau} \boxtimes \bbT\an_{X/S} \big)^\vee[1], \]
	by \cite[Lemma 8.4]{Porta_Yu_Non-archimedean_quantum_K-invariants}.
	
	Note that $\mathfrak M(X/S, \tau, 0)$ is derived lci, since all of its derived fibers over $\bA^1_k$ are.
		Therefore, its virtual fundamental class
	\[ [ \mathfrak M(X/S, \tau, 0) ] \in \HBM_{2d}( \oM_\tau \times X \times \bA^1_k / \bA^1_k,\bbQ(d))\]
	is well-defined, where $d \coloneqq g(\tau)\dim X/S$.
	By $\bA^1_k$-invariance, the base change homomorphisms $0^*$ and $1^*$ induce the same identification $\HBM_{2d}( \oM_\tau \times X \times \bA^1_k / \bA^1_k,\bbQ(d)) \xrightarrow{\sim} \HBM_{2d}( \oM_\tau \times X,\bbQ(d))$, so we obtain
	\[ [\RoM(X/S,\tau,0)] = 1^* ([ \mathfrak M(X/S,\tau,0)]) = 0^*( [\mathfrak M(X/S,\tau,0)] ) = [\mathfrak N] . \]
	To compute $[\mathfrak N]$, we observe that $\anL_j[-2]$ is a vector bundle over $\oM_\tau \times X$, and set
	\[ V \coloneqq \Spec_{\oM_\tau \times X}( \Sym_{\cO_{\oM_\tau \times X}}(\anL_j[-2]) ) . \]
	Then we have a derived fiber product
	\[ \begin{tikzcd}
		\mathfrak N \rar \dar & \oM_\tau \times X \dar{0} \\
		\oM_\tau \times X \rar{0} & V .
	\end{tikzcd} \]
	Using the excess intersection formula of \cite[Proposition 3.15]{Khan_Virtual_fundamental_classes}, we obtain
	\[ [\mathfrak N] = c_{g(\tau)\dim X/S}\big(\rR^1\pi_*\cO_{\oC_\tau}\boxtimes\bbT\an_{X/S}\big)\cap[\oM(X/S,\tau,0)],\]
	completing the proof.
\end{proof}

\subsection{Products} \label{sec:products_vfc}

\begin{theorem} \label{thm:products_vfc}
	Let $(\tau_1, \beta_1)$ and $(\tau_2, \beta_2)$ be two A-graphs.
	We have
	\[\RoM(X/S,\tau_1\sqcup\tau_2,\beta_1\sqcup\beta_2)]=[\RoM(X/S, \tau_1, \beta_1)] \boxtimes [\RoM(X/S, \tau_2, \beta_2)].\]
\end{theorem}
\begin{proof}
	Let $a\colon \RoM(X/S, \tau_1, \beta_1)\to S$ and $b\colon \RoM(X/S, \tau_2, \beta_2)\to S$ be the structure morphisms.
	By \cref{thm:geometric_relations}(\ref{thm:geometric_relations:products}), the functoriality and base change of virtual fundamental classes (Propositions \ref{prop:vfc_functoriality} and \ref{prop:vfc_base_change}), and \cref{prop:properties_of_products}(\ref{prop:properties_of_products:external_and_composition}), we obtain
	\begin{align*}
		\RoM(X/S,\tau_1\sqcup\tau_2,\beta_1\sqcup\beta_2)]&=a^*[\RoM(X/S, \tau_2, \beta_2)] \circ [\RoM(X/S, \tau_1, \beta_1)]\\
		&=b^*[\RoM(X/S, \tau_1, \beta_1)] \circ [\RoM(X/S, \tau_2, \beta_2)]\\
		&=[\RoM(X/S, \tau_1, \beta_1)] \boxtimes [\RoM(X/S, \tau_2, \beta_2)].
	\end{align*}
\end{proof}

\subsection{Cutting edges} \label{sec:cutting_edges_vfc}

Let $(\sigma,\beta)$ be an A-graph obtained from $(\tau,\beta)$ by cutting an edge $e$ of $\tau$.
Let $v,w$ be the two tails of $\sigma$ created by the cut.
The edge $e$ gives a section $s_e$ of the universal curve $\oCpre_\tau\to\oMpre_\tau$, corresponding to the nodes.
By \cite[Theorem 5.2]{Porta_Yu_Non-archimedean_quantum_K-invariants}, we have a derived pullback diagram
\[ \begin{tikzcd}
\RoM(X/S, \tau, \beta) \rar{c} \dar{\ev_e} & \RoM(X/S, \sigma, \beta) \dar{\ev_v \times \ev_w} \\
X \rar{\Delta} & X \times_S X
\end{tikzcd} \]
where $\Delta$ is the diagonal map, $\ev_e$ is evaluation at the section $s_e$, and $c$ is induced by cutting the domain curves at $s_e$.

\begin{theorem} \label{thm:cutting_edges_vfc}
	We have
	\[[\RoM(X/S, \tau, \beta)]=\Delta^![\RoM(X/S, \sigma, \beta)].\]
\end{theorem}
\begin{proof}
	Let
	\[a_\tau\colon\RoM(X/S, \tau, \beta)\to S, \quad a_\sigma\colon\RoM(X/S, \sigma, \beta)\to S\]
	denote the structure morphisms.
	By the functoriality of the Gysin map, we have
	\begin{multline*}
	[\RoM(X/S, \tau, \beta)]=a_\tau^!(1)=(a_\sigma \circ c)^!(1)=c^!(a_\sigma^!(1))\\
	=c^![\RoM(X/S, \sigma, \beta)]=\Delta^![\RoM(X/S, \sigma, \beta)],
	\end{multline*}
	where the last equality is only notational (see \cref{rem:Gysin}).
\end{proof}

\subsection{Universal curve} \label{sec:universal_curve_vfc}

Let $(\sigma,\beta)$ be an A-graph obtained from $(\tau,\beta)$ by forgetting a tail $t$ attached to a vertex $w$.
Let $\oCpre_w\to\oMpre_\sigma$ be the universal curve corresponding to $w$.
By \cite[\S 5.3]{Porta_Yu_Non-archimedean_quantum_K-invariants}, we have a derived pullback diagram
\[ \begin{tikzcd}
\RoM( X/S, \tau, \beta ) \rar{\pi} \dar & \RoM( X/S, \sigma, \beta ) \dar \\
\oCpre_w \rar & \oMpre_\sigma.
\end{tikzcd} \]

\begin{theorem} \label{thm:universal_curve_vfc}
	We have
	\[[\RoM(X/S, \tau, \beta)]=\pi^![\RoM(X/S, \sigma, \beta)].\]
\end{theorem}
\begin{proof}
	Let
	\[a_\tau\colon\RoM(X/S, \tau, \beta)\to S, \quad a_\sigma\colon\RoM(X/S, \sigma, \beta)\to S\]
	denote the structure morphisms.
	By the functoriality of the Gysin map, we have
	\[[\RoM(X/S, \tau, \beta)]=a_\tau^!(1)=(a_\sigma \circ \pi)^!(1)=\pi^!(a_\sigma^!(1))=\pi^![\RoM(X/S, \sigma, \beta)].\]
\end{proof}

\subsection{Forgetting tails} \label{sec:forgetting_tails_vfc}

Notation as in \cref{sec:universal_curve_vfc}, by \cite[\S 5.4]{Porta_Yu_Non-archimedean_quantum_K-invariants}, we have a commutative diagram where all the squares are derived pullbacks:
\[ \begin{tikzcd}
	\RoM(X/S, \tau, \beta) \rar{\Psi} \dar{\lambda} & \oM_\tau \times_{\oM_\sigma}  \RoM(X/S, \sigma, \beta) \rar{\eta} \dar{\mu} & \RoM(X/S, \sigma, \beta) \dar \\
	\oCpre_w \rar{\rho} \arrow{rd} & \oM_\tau \times_{\oM_\sigma} \oMpre_\sigma \dar \rar & \oMpre_\sigma \dar \\ {} & \oM_\tau \rar{\Phi} & \oM_\sigma .
\end{tikzcd} \]

\begin{theorem} \label{thm:forgetting_tails_vfc}
	We have
	\[\Psi_*[\RoM(X/S, \tau, \beta)]=\Phi^![\RoM(X/S, \sigma, \beta)],\]
	where $\Phi\colon\oM_\tau\to\oM_\sigma$ is the forgetting-tails map for pointed stable curves.
\end{theorem}
\begin{proof}
	Let $(\oMpre_\sigma)^0\subset\oMpre_\sigma$ be the dense open substack where the pointed curve $\oCpre_w$ is stable.
	Then $\rho$ is an isomorphism over $(\oMpre_\sigma)^0$.
		So we have
	\begin{equation} \label{eq:rho_forgetting_tails_vfc}
		\rho_*[\oCpre_w]=\big[\oM_\tau \times_{\oM_\sigma} \oMpre_\sigma\big].
	\end{equation}
		Next we compute
	\begin{align*}
	\Psi_*[\RoM(X/S, \tau, \beta)] &= \Psi_*\big(\lambda^!(1)\circ[\oCpre_w]\big) && \text{by \cref{prop:vfc_functoriality}} \\
	&=\Psi_*\big((\rho^*\mu^!(1))\circ[\oCpre_w]\big) && \text{by \cref{prop:vfc_base_change}} \\
	&=\mu^!(1)\circ\rho_*[\oCpre_w] && \text{by the projection formula (\cite[\S 2.3.4]{Khan_Virtual_fundamental_classes})} \\
	&=\mu^!(1)\circ\big[\oM_\tau \times_{\oM_\sigma} \oMpre_\sigma\big] && \text{by \eqref{eq:rho_forgetting_tails_vfc}} \\
	&=\big[\oM_\tau \times_{\oM_\sigma} \RoM(X/S, \sigma, \beta)\big] && \text{by \cref{prop:vfc_functoriality}} \\
	&=\eta^![\RoM(X/S, \sigma, \beta)\big] && \text{by \cref{prop:vfc_functoriality}} \\
	&=\Phi^![\RoM(X/S, \sigma, \beta)] && \text{by \cref{rem:Gysin}},
	\end{align*}
	completing the proof.
\end{proof}

\subsection{Contracting edges} \label{sec:contracting_edges_vfc}

Let $(\sigma,\beta)$ be an A-graph where $\sigma$ is obtained from a modular graph $\tau$ by contracting an edge (possibly a loop) $e$.
Let $\beta_j, j\in J$ be all possible curve classes on $\tau$ that project to $\beta$ under the contraction.
By \cite[Lemma 5.20]{Porta_Yu_Non-archimedean_quantum_K-invariants}, we have a commutative diagram where all the squares are derived pullbacks:
\[\begin{tikzcd}
\coprod_j \RoM(X/S,\tau,\beta_j) \rar{\Psi} \dar{\lambda}& \oM_\tau\times_{\oM_\sigma}\RoM(X/S,\sigma,\beta) \rar{\eta} \dar{\mu} & \RoM(X/S,\sigma,\beta)\dar\\
\oMpre_\tau \rar{\rho} \drar & \oM_\tau \times_{\oM_\sigma} \oMpre_\sigma \rar \dar & \oMpre_\sigma \dar \\
& \oM_\tau \rar{\Phi} & \oM_\sigma
\end{tikzcd}\]

\begin{theorem} \label{thm:contracting_edges_vfc}
	We have
	\[\sum_j\Psi_*[\RoM(X/S, \tau, \beta_j)]=\Phi^![\RoM(X/S, \sigma, \beta)],\]
	where $\Phi\colon\oM_\tau\to\oM_\sigma$ is the contracting-edges map for pointed stable maps.
\end{theorem}
\begin{proof}
	Let $v_1, v_2$ be the two (possibly identical) endpoints of $e$, and $v_0$ the vertex of $\sigma$ onto which the edge $e$ is contracted.
	Let $\oCpre_{v_1}\to\oMpre_\tau$, $\oCpre_{v_2}\to\oMpre_\tau$ and $\oCpre_{v_0}\to\oMpre_\sigma$ be the corresponding universal curves.
	Let $(\oMpre_\tau)^0\subset\oMpre_\tau$ be the dense open substack where the pointed curves $\oCpre_{v_1}$ and $\oCpre_{v_2}$ are both stable.
	Let $(\oMpre_\sigma)^0\subset\oMpre_\sigma$ be the dense open substack where the pointed curve $\oCpre_{v_0}$ is stable.
	Then $\rho$ induces an isomorphism between $(\oMpre_\tau)^0$ and $\oM_\tau\times_{\oM_\sigma}(\oMpre_\sigma)^0$.
		So we have
	\begin{equation} \label{eq:rho_contracting_edges_vfc}
		\rho_*[\oMpre_\tau]=\big[\oM_\tau \times_{\oM_\sigma} \oMpre_\sigma\big].
	\end{equation}
		Then we compute
	\begin{align*}
		\sum_j \Psi_*[\RoM(X/S, \tau, \beta_j)]&=\Psi_*\bigg[\coprod_j\RoM(X/S, \tau, \beta_j)\bigg] \\
		&=\Psi_*\big(\lambda^!(1)\circ[\oMpre_\tau]\big) && \text{by \cref{prop:vfc_functoriality}} \\
		&=\Psi_*\big((\rho^*\mu^!(1))\circ[\oMpre_\tau]\big) && \text{by \cref{prop:vfc_base_change}} \\
		&=\mu^!(1)\circ\rho_*[\oMpre_\tau] && \text{by the projection formula (\cite[\S 2.3.4]{Khan_Virtual_fundamental_classes})} \\
		&=\mu^!(1)\circ\big[\oM_\tau \times_{\oM_\sigma} \oMpre_\sigma\big] && \text{by \eqref{eq:rho_contracting_edges_vfc}} \\
		&=\big[\oM_\tau \times_{\oM_\sigma} \RoM(X/S, \sigma, \beta)\big] && \text{by \cref{prop:vfc_functoriality}} \\
		&=\eta^![\RoM(X/S, \sigma, \beta)] && \text{by \cref{prop:vfc_functoriality}} \\
		&=\Phi^![\RoM(X/S, \sigma, \beta)] && \text{by \cref{rem:Gysin},}
	\end{align*}
	completing the proof.
\end{proof}

\section{Non-archimedean Gromov-Witten invariants} \label{sec:numerical_GW}

In order to obtain enumerative invariants, we consider the case where $S=\Spf k$, and we drop $/S$ from the notations.
We will assume from now on that the ground field $k$ has discrete valuation and residue characteristic zero.
The properness of the derived moduli stack $\RoM(X,\tau,\beta)$ can be ensured by the non-archimedean Gromov compactness theorem provided that $X$ is endowed with a non-archimedean Kähler structure (see \cite{Yu_Gromov_compactness}).
Consider the evaluation maps and the stabilization map of domain curves.
\[\begin{tikzcd}
	\RoM(X,\tau,\beta) \rar{\ev_i} \dar{\st} & X\\
	\oM_\tau 
\end{tikzcd}\]
\begin{definition} \label{def:GW}
	The \emph{non-archimedean Gromov-Witten invariants} are the collection of linear maps
	\begin{align*}
		I^X_{\tau,\beta}\colon \rH^*(X,\bbQ(*))^{\otimes n} &\longto \rH^*(\oM_\tau,\bbQ(*)),\\
		\bigotimes_{i\in T_\tau} a_i &\longmapsto\PD\inv\st_*\bigg(\Big(\underset{i\in T_\tau}{\cup}\ev_i^*(a_i)\Big)\cap [\RoM(X,\tau,\beta)]\bigg),
	\end{align*}
	where $\PD$ denotes the Poincaré duality isomorphism (see \cite[Theorem 2.23]{Khan_Virtual_fundamental_classes}).
	In the case where $\tau$ consists of a single vertex of genus $g$ and $n$ tails, we write the invariants as
	\[I^X_{g,n,\beta}\colon \rH^*(X,\bbQ(*))^{\otimes n} \longto \rH^*(\oM_{g,n},\bbQ(*)).\]
\end{definition}

Using the functorial properties of motivic cohomology and motivic Borel-Moore homology established in \cref{prop:properties_of_products}, we deduce the following equalities of non-archimedean Gromov-Witten invariants from the corresponding relations of virtual fundamental classes proved in \cref{sec:axioms}.

\subsection{Mapping to a point}

\begin{proposition} \label{prop:mapping_to_a_point_GW}
	In the setting of \cref{sec:mapping_to_a_point}, for any $a_i\in H^*(X,\bbQ(*))$, $i\in T_\tau$, we have
	\[I_{\tau,0}^X({\textstyle\bigotimes}_i a_i) = \PD^{-1}\st_*\Big((\cup_i\ev_i^*a_i)\cap c_{g(\tau)\dim X}\big(\rR^1\pi_*\cO_{\oC_\tau}\boxtimes\bbT\an_X\big)\Big).\]
\end{proposition}
	
\subsection{Products}

\begin{proposition} \label{prop:products_GW}
	In the setting of \cref{sec:products_vfc}, for any $a_i\in H^*(X,\bbQ(*))$, $i\in T_{\tau_1}$ and $b_j\in H^*(X,\bbQ(*))$, $j\in T_{\tau_2}$, we have
	\[I^X_{\tau_1\sqcup\tau_2,\beta_1\sqcup\beta_2}\big(({\textstyle\bigotimes}_i a_i)\otimes({\textstyle\bigotimes_j} b_j)\big)=I^X_{\tau_1,\beta_1}\big({\textstyle\bigotimes}_i a_i\big)\boxtimes I^X_{\tau_2,\beta_2}\big({\textstyle\bigotimes_j} b_j\big).\]
\end{proposition}
\begin{proof}
	For each $i\in T_{\tau_1}$ and $j\in T_{\tau_2}$, we denote the evaluation maps by
	\begin{align*}
		\ev^1_i&\colon\RoM(X,\tau_1,\beta_1)\longrightarrow X,\\
		\ev^2_j&\colon\RoM(X,\tau_2,\beta_2)\longrightarrow X.
	\end{align*}
	Using \cref{thm:products_vfc}, we compute
	\begin{align*}
		& I^X_{\tau_1\sqcup\tau_2,\beta_1\sqcup\beta_2}\big(({\textstyle\bigotimes}_i a_i)\otimes({\textstyle\bigotimes_j b_j)\big)} \\ = {} & \PD^{-1}\st_*\Big(\big((\cup_i\ev_i^{1*} a_i)\boxtimes(\cup_j\ev_j^{2*}b_j)\big)\cap\big([\RoM(X,\tau_1,\beta_1)]\boxtimes[\RoM(X,\tau_2,\beta_2)]\big)\Big) \\
		= {} & \PD^{-1}\st_*\Big(\big((\cup_i\ev_i^{1*} a_i)\cap[\RoM(X,\tau_1,\beta_1)]\big)\boxtimes\big((\cup_j\ev_j^{2*}b_j)\big)\cap[\RoM(X,\tau_2,\beta_2)]\big)\Big) \\
		= {} & \PD^{-1}\Big(\st_*\big((\cup_i\ev_i^{1*} a_i)\cap[\RoM(X,\tau_1,\beta_1)]\big)\boxtimes\st_*\big((\cup_j\ev_j^{2*}b_j)\big)\cap[\RoM(X,\tau_2,\beta_2)]\big)\Big) \\
		= {} & \PD^{-1}\st_*\big((\cup_i\ev_i^{1*} a_i)\cap[\RoM(X,\tau_1,\beta_1)]\big)\boxtimes\PD^{-1}\st_*\big((\cup_j\ev_j^{2*}b_j)\big)\cap[\RoM(X,\tau_2,\beta_2)]\big) \\
		= {} & I^X_{\tau_1,\beta_1}\big({\textstyle\bigotimes}_i a_i\big)\boxtimes I^X_{\tau_2,\beta_2}\big({\textstyle\bigotimes_j} b_j\big),
	\end{align*}
	where the second equality follows from \cref{prop:properties_of_products}(\ref{prop:properties_of_products:compatibility_external_cap}), the third from \cref{prop:properties_of_products}(\ref{prop:properties_of_product:external_product_pushforward}), and the fourth follows also from \cref{prop:properties_of_products}(\ref{prop:properties_of_products:compatibility_external_cap}), since the Poincaré duality isomorphism is given by the cap product.
\end{proof}
	
\subsection{Cutting edges}

\begin{proposition} \label{prop:cutting_edges_GW}
	In the setting of \cref{sec:cutting_edges_vfc}, for any $a_i\in H^*(X,\bbQ(*)), i\in T_\tau$, we have
	\[d_! I^X_{\tau,\beta}\big({\textstyle\bigotimes}_i a_i\big)=I^X_{\sigma,\beta}\big(\Delta_! 1_X\otimes({\textstyle\bigotimes}_i a_i)\big),\]
	where $d\colon\oM_\tau\to\oM_\sigma$ is induced by cutting the edge.
\end{proposition}
\begin{proof}
	Consider the commutative diagram
	\[ \begin{tikzcd}
	\oM_\tau \rar{d} & \oM_\sigma\\
	\RoM(X, \tau, \beta) \uar{\st_\tau} \rar{c} \dar{\ev_e} & \RoM(X, \sigma, \beta) \uar{\st_\sigma} \dar{\ev_v \times \ev_w} \\
	X \rar{\Delta} & X \times_S X.
	\end{tikzcd} \]
	For each $i\in T_\tau$, we denote the evaluation maps by
	\begin{align*}
		\ev^\tau_i&\colon\RoM(X,\tau,\beta)\longrightarrow X,\\
		\ev^\sigma_i&\colon\RoM(X,\sigma,\beta)\longrightarrow X.
	\end{align*}
	We have $\ev^\tau_i=\ev^\sigma_i\circ c$.
	We obtain
	\begin{align*}
		d_! I^X_{\tau,\beta}\big({\textstyle\bigotimes}_i a_i\big)
		&= d_! \PD^{-1}\st_{\tau*} \big((\cup_i\ev^{\tau*}_i a_i)\cap[\RoM(X,\tau,\beta)]\big)\\
		&= \PD^{-1}d_*\st_{\tau*} \Big(\big(\ev_e^* 1_X\cup(\cup_i\ev^{\tau*}_i a_i)\big)\cap[\RoM(X,\tau,\beta)]\Big)\\
		&= \PD^{-1}\st_{\sigma*} c_*\Big(\ev_e^* 1_X\cap\big( c^*(\cup_i\ev^{\sigma*}_i a_i)\cap c^![\RoM(X,\sigma,\beta)]\big)\Big) \\ & && \text{\hspace{-5em} by \cref{prop:vfc_functoriality}} \\
		&= \PD^{-1}\st_{\sigma*} c_*\Big(\ev_e^* 1_X\cap c^!\big( (\cup_i\ev^{\sigma*}_i a_i)\cap [\RoM(X,\sigma,\beta)]\big)\Big)\\
		& && \text{\hspace{-5em} by \cref{prop:properties_of_products}(\ref{prop:properties_of_products:pullback_cap})} \\
		&= \PD^{-1}\st_{\sigma*} \Big(c_!\ev_e^* 1_X\cap \big( (\cup_i\ev^{\sigma*}_i a_i)\cap [\RoM(X,\sigma,\beta)]\big)\Big) \\
		& && \text{\hspace{-5em} by \cref{prop:properties_of_products}(\ref{prop:properties_of_products:projection_formula_cap_Gysin})} \\
		&= \PD^{-1}\st_{\sigma*} \Big( (\ev_v\times\ev_w)^*\Delta_! 1_X\cap \big( (\cup_i\ev^{\sigma*}_i a_i)\cap [\RoM(X,\sigma,\beta)]\big)\Big)\\
		& && \text{\hspace{-5em} by \cref{prop:Gysin_pushforward_base_change}} \\
		&= \PD^{-1}\st_{\sigma*} \Big(\big((\ev_v\times\ev_w)^*\Delta_! 1_X\cup(\cup_i\ev^{\sigma*}_i a_i)\big)\cap [\RoM(X,\sigma,\beta)]\big)\Big)\\
		&= I^X_{\sigma,\beta}\big(\Delta_! 1_X\otimes({\textstyle\bigotimes}_i a_i)\big),
	\end{align*}
	completing the proof.
\end{proof}
	
\subsection{Forgetting tails}

\begin{proposition} \label{prop:forgetting_tails_GW}
	In the setting of \cref{sec:forgetting_tails_vfc}, for any $a_i\in H^*(X,\bbQ(*)), i\in T_\sigma$, we have
	\[I^X_{\tau,\beta}\big(({\textstyle\bigotimes}_i a_i)\otimes 1_X\big)=\Phi^*I^X_{\sigma,\beta}\big({\textstyle\bigotimes}_i a_i\big).\]
\end{proposition}
\begin{proof}
	Consider the commutative diagram
	\[\begin{tikzcd}
		\RoM(X, \tau, \beta) \rar{\Psi} \dar{\lambda}\arrow{rdd}[near end, swap]{\st_\tau} & \oM_\tau \times_{\oM_\sigma}  \RoM(X, \sigma, \beta) \rar{\eta} \dar{\mu} & \RoM(X, \sigma, \beta) \dar{\kappa} \arrow[bend left=40]{dd}{\st_\sigma} \\
		\oCpre_w \arrow[crossing over]{r}[near start]{\rho} & \oM_\tau \times_{\oM_\sigma} \oMpre_\sigma \dar{\xi} \rar & \oMpre_\sigma \dar{\theta} \\
		{} & \oM_\tau \rar{\Phi} & \oM_\sigma .
	\end{tikzcd} \]
	For each $i\in T_\sigma$, we denote the evaluation maps by
	\begin{align*}
	\ev^\tau_i&\colon\RoM(X,\tau,\beta)\longrightarrow X,\\
	\ev^\sigma_i&\colon\RoM(X,\sigma,\beta)\longrightarrow X,
	\end{align*}
	where $\ev^\tau_i=\ev^\sigma_i\circ\eta\circ\Psi$.
	Let $t\in T_\tau$ denote the forgotten tail.
	We obtain
	\begin{align*}
		I^X_{\tau,\beta}\big(({\textstyle\bigotimes}_i a_i)\otimes 1_X\big)
		&= \PD^{-1}\st_{\tau*} \Big(\big((\cup_i\ev^{\tau*}_i a_i)\cup\ev_t^* 1_X\big)\cap[\RoM(X,\tau,\beta)]\Big)\\
		&= \PD^{-1}\xi_*\mu_*\Psi_*\big(\Psi^*\eta^*(\cup_i\ev^{\sigma*}_i a_i)\cap[\RoM(X,\tau,\beta)]\big)\\
		&= \PD^{-1}\xi_*\mu_*\big(\eta^*(\cup_i\ev^{\sigma*}_i a_i)\cap\Psi_*[\RoM(X,\tau,\beta)]\big) && \text{by \cref{prop:properties_of_products}(\ref{prop:properties_of_products:projection_formula_cap})} \\
		&= \PD^{-1}\xi_*\mu_*\big(\eta^*(\cup_i\ev^{\sigma*}_i a_i)\cap\Phi^![\RoM(X,\sigma,\beta)]\big) && \text{by \cref{thm:forgetting_tails_vfc}} \\
		&= \PD^{-1}\xi_*\mu_*\big(\eta^*(\cup_i\ev^{\sigma*}_i a_i)\cap\eta^![\RoM(X,\sigma,\beta)]\big) && \text{by \cref{rem:Gysin}} \\
		&= \PD^{-1}\xi_*\mu_*\eta^!\big((\cup_i\ev^{\sigma*}_i a_i)\cap[\RoM(X,\sigma,\beta)]\big) && \text{by \cref{prop:properties_of_products}(\ref{prop:properties_of_products:pullback_cap})} \\
		&= \PD^{-1}\Phi^!\theta_*\kappa_*\big((\cup_i\ev^{\sigma*}_i a_i)\cap[\RoM(X,\sigma,\beta)]\big) && \text{by \cref{prop:Gysin_pullback_base_change}} \\
		&= \Phi^*\PD^{-1}\st_{\sigma*}\big((\cup_i\ev^{\sigma*}_i a_i)\cap[\RoM(X,\sigma,\beta)]\big)\\
		&= \Phi^*I^X_{\sigma,\beta}\big({\textstyle\bigotimes}_i a_i\big),
	\end{align*}
	completing the proof.
\end{proof}
	
\subsection{Contracting edges}

\begin{proposition} \label{prop:contracting_edges_GW}
	In the setting of \cref{sec:contracting_edges_vfc}, for any $a_i\in H^*(X,\bbQ(*)), i\in T_\sigma$, we have
	\[\sum_j I^X_{\tau,\beta_j}\big({\textstyle\bigotimes}_i a_i\big)=\Phi^* I^X_{\sigma,\beta}\big({\textstyle\bigotimes}_i a_i\big).\]
\end{proposition}
\begin{proof}
	For each $i\in T_\sigma$, we denote the evaluation maps by
	\begin{align*}
		\ev^\tau_i&\colon\coprod_j \RoM(X/S,\tau,\beta_j)\longrightarrow X,\\
		\ev^\sigma_i&\colon\RoM(X,\sigma,\beta)\longrightarrow X,
	\end{align*}
	where $\ev^\tau_i=\ev^\sigma_i\circ\eta\circ\Psi$.
	We obtain
	\begin{align*}
		\sum_j I^X_{\tau,\beta_j}\big({\textstyle\bigotimes}_i a_i\big)
		&= \sum_j\PD^{-1}\st_{\tau*} \big((\cup_i\ev^{\tau*}_i a_i)\cap[\RoM(X,\tau,\beta_j)]\big)\\
		&= \sum_j\PD^{-1}\xi_*\mu_*\Psi_*\big(\Psi^*\eta^*(\cup_i\ev^{\sigma*}_i a_i)\cap[\RoM(X,\tau,\beta_j)]\big)\\
		&= \sum_j \PD^{-1}\xi_*\mu_*\big(\eta^*(\cup_i\ev^{\sigma*}_i a_i)\cap\Psi_*[\RoM(X,\tau,\beta)]\big)\\
		&= \PD^{-1}\xi_*\mu_*\big(\eta^*(\cup_i\ev^{\sigma*}_i a_i)\cap\Phi^![\RoM(X,\sigma,\beta)]\big)\\
		&= \PD^{-1}\xi_*\mu_*\big(\eta^*(\cup_i\ev^{\sigma*}_i a_i)\cap\eta^![\RoM(X,\sigma,\beta)]\big)\\
		&= \PD^{-1}\xi_*\mu_*\eta^!\big((\cup_i\ev^{\sigma*}_i a_i)\cap[\RoM(X,\sigma,\beta)]\big)\\
		&= \PD^{-1}\Phi^!\theta_*\kappa_*\big((\cup_i\ev^{\sigma*}_i a_i)\cap[\RoM(X,\sigma,\beta)]\big)\\
		&= \Phi^*\PD^{-1}\st_{\sigma*}\big((\cup_i\ev^{\sigma*}_i a_i)\cap[\RoM(X,\sigma,\beta)]\big)\\
		&= \Phi^*I^X_{\sigma,\beta}\big({\textstyle\bigotimes}_i a_i\big),
	\end{align*}
	where the justifications of the equalities are similar to those in the proof of \cref{prop:forgetting_tails_GW}.
\end{proof}

\section{Non-archimedean Gromov-Witten invariants with naive tangencies} \label{sec:tangencies}

In this section, we introduce non-archimedean Gromov-Witten invariants with naive tangencies, a generalization of those considered in \cref{sec:numerical_GW}.
We refer to the introduction (\cref{sec:intro}) for a discussion of the underlying geometric idea.

Fix a proper smooth \kanal space $X$, an $A$-graph $(\tau,\beta)$, and a tail vertex $i\in T_\tau$.
First we would like to generalize the evaluation map
\[\ev_i\colon\RoM(X,\tau,\beta)\longto X\]
in order to take into account the jets at the $i$-th marked point.
Let $m_i$ be a positive integer.
Let $s_i\colon\oMpre_\tau\to\oCpre_\tau$ be the $i$-th section of the universal curve, $\cI_i$ the ideal sheaf on $\oCpre_\tau$ cutting out the image of $s_i$, and $(\oCpre_\tau)_{(s_i^{m_i})}$ the closed substack given by the $m_i$-th power of the ideal $\cI_i$.
Let $\ev_i^{m_i}$ denote the composition of the inclusion
\[\RoM(X,\tau,\beta) \longhookrightarrow \bfMap_{\oMpre_\tau}\bigl(\oCpre_\tau, X\times\oMpre_\tau\bigr).\]
and the restriction
\[\bfMap_{\oMpre_\tau}\bigl(\oCpre_\tau, X\times\oMpre_\tau\bigr) \longto \bfMap_{\oMpre_\tau}\Bigl((\oCpre_\tau)_{(s_i^{m_i})}, X\times \oMpre_\tau\Bigr)\eqqcolon X_{i,\tau}^{m_i}.\]
We call
\[\ev_i^{m_i}\colon\RoM(X,\tau,\beta)\longrightarrow X_{i,\tau}^{m_i}\]
\emph{the evaluation map of the $i$-th marked point of order $m_i$}.

Given any closed analytic subspace $Z_i\subset X$ such that the inclusion is lci, let
\[Z_{i,\tau}^{m_i}\coloneqq\bfMap_{\oMpre_\tau}\Bigl((\oCpre_\tau)_{(s_i^{m_i})}, Z_i\times \oMpre_\tau\Bigr).\]
By \cite[Lemma 9.1]{Porta_Yu_Non-archimedean_quantum_K-invariants}, the derived \kanal stack $X_{i,\tau}^{m_i}$ is smooth over $\oMpre_\tau$, in particular underived; and the map $\zeta_i\colon Z_{i,\tau}^{m_i}\to X_{i,\tau}^{m_i}$ is derived lci.
Therefore, by \cref{def:virtual_fundamental_class}, we have a virtual fundamental class $[Z_{i,\tau}^{m_i}]\in\HBM_*(Z_{i,\tau}^{m_i},\bbQ(*))$.

\begin{definition} \label{def:GW_with_tangency}
	Given any A-graph $(\tau,\beta)$, $\mathbf m_i=(m_i)_{i\in T_\tau}$ with $m_i\in\bbN_{>0}$, and $\bZ=(Z_i)_{i\in T_\tau}$ with $Z_i\subset X$ closed lci, we define the associated \emph{Gromov-Witten invariants of $X$ with naive tangencies}
	\[
	I^X_{\tau,\beta,\mathbf m}(\bZ)\coloneqq \PD^{-1}\st_*\bigg(\Big(\underset{i\in T_\tau}{\cup}(\ev_i^{m_i})^*\PD^{-1}\zeta_{i*}[Z_{i,\tau}^{m_i}]\Big)\cap[\RoM(X,\tau,\beta)]\bigg) \in H^*(\oM_\tau,\bbQ(*)).\]
\end{definition}

The Gromov-Witten invariants with naive tangencies satisfy the following list of natural relations exactly parallel to the previous section.

\subsection{Mapping to a point}

\begin{proposition} \label{prop:mapping_to_a_point_with_tangency}
	In the setting of \cref{sec:mapping_to_a_point}, for any $\mathbf m=(m_i)_{i\in T_\tau}$ and $\bZ=(Z_i)_{i\in T_\tau}$ as in \cref{def:GW_with_tangency}, we have
	\[I^X_{\tau,\beta,\mathbf m}(\bZ) = \PD^{-1}\st_*\bigg(\Big(\underset{i\in T_\tau}{\cup}(\ev_i^{m_i})^*\PD^{-1}\zeta_{i*}[Z_{i,\tau}^{m_i}]\Big)\cap c_{g(\tau)\dim X}\big(\rR^1\pi_*\cO_{\oC_\tau}\boxtimes\bbT\an_X\big)\bigg).\]
\end{proposition}

\subsection{Products}

\begin{proposition}
	In the setting of \cref{sec:products_vfc}, for any $\mathbf m=(m_i)_{i\in T_{\tau_1}}$, $\bZ=(Z_i)_{i\in T_{\tau_1}}$, $\mathbf n=(n_j)_{j\in T_{\tau_2}}$, $\bW=(W_j)_{j\in T_{\tau_2}}$ as in \cref{def:GW_with_tangency}, we have
	\[I^X_{\tau_1\sqcup\tau_2,\beta_1\sqcup\beta_2,\mathbf m\sqcup\mathbf n}(\bZ\sqcup\bW)=I^X_{\tau_1,\beta_1,\mathbf m}(\bZ)\times I^X_{\tau_2,\beta_2,\mathbf n}(\bW).\]
\end{proposition}

\subsection{Cutting edges}

\begin{proposition}
	In the setting of \cref{sec:cutting_edges_vfc}, for any $\mathbf m=(m_i)_{i\in T_\tau}$ and $\bZ=(Z_i)_{i\in T_\tau}$ as in \cref{def:GW_with_tangency}, let
	\[\mathbf m'\coloneqq\mathbf m\sqcup(m_v=1, m_w=1).\]
	We have
	\[d_*I^X_{\tau,\beta,\mathbf m}(\bZ)=I^X_{\sigma,\beta,\mathbf m'}(\Delta_*[X]\sqcup \bZ),\]
	where $d\colon\oM_\tau\to\oM_\sigma$ is induced by cutting the edge $e$ of $\tau$, and $I^X_{\sigma,\beta,\mathbf m'}(\Delta_*[X]\sqcup \bZ)$ means
	\[\PD^{-1}\st_*\bigg(\big((\ev_v\times\ev_w)^*\Delta_*[X]\big)\cup \Big(\underset{i\in T_\tau}{\cup}(\ev_i^{m_i})^*\PD^{-1}\zeta_{i*}[Z_{i,\tau}^{m_i}]\Big)\cap[\RoM(X,\tau,\beta)]\bigg).\]
\end{proposition}

\subsection{Forgetting tails}

\begin{proposition}
	In the setting of \cref{sec:forgetting_tails_vfc}, for any $\mathbf m=(m_i)_{i\in T_\sigma}$ and $\bZ=(Z_i)_{i\in T_\sigma}$ as in \cref{def:GW_with_tangency}, and any positive integer $m_t$, let
	\[\mathbf m'\coloneqq\mathbf m\sqcup(m_t).\]
	We have	
	\[I^X_{\tau,\beta,\mathbf m'}(\bZ\sqcup(X))=\Phi^*I^X_{\sigma,\beta,\mathbf m}(\bZ).\]
\end{proposition}

\subsection{Contracting edges}

\begin{proposition}
	In the setting of \cref{sec:contracting_edges_vfc}, for any $\mathbf m=(m_v)_{v\in T_{\sigma}}$ and $\bZ=(Z_v)_{v\in T_{\sigma}}$ as in \cref{def:GW_with_tangency}, we have
	\[\sum_j I^X_{\tau,\beta_j,\mathbf m}(\bZ)=\Phi^* I^X_{\sigma,\beta,\mathbf m}(\bZ).\]
\end{proposition}

\bibliographystyle{plain}
\bibliography{dahema}

\end{document}